\documentclass[11pt]{article}
\usepackage{amsthm,amsmath,amssymb}
\usepackage[margin=25mm]{geometry}
\usepackage{mathtools}
\usepackage{enumitem}
\usepackage{hyperref}
\usepackage{xcolor}
\usepackage{bm}
\usepackage[normalem]{ulem}
\mathtoolsset{showonlyrefs=true}

\allowdisplaybreaks[4]

\theoremstyle{definition}
\newtheorem{exm}{Example}[section]

\theoremstyle{plain}
\newtheorem{lem}[exm]{Lemma}
\newtheorem{prop}[exm]{Proposition}
\newtheorem{thm}[exm]{Theorem}
\newtheorem{cor}[exm]{Corollary}
\newtheorem{assum}[exm]{Assumption}

\DeclareMathOperator{\supp}{supp}

\newcommand{\NN}{\mathbb{N}}
\newcommand{\ZN}{\mathbb{Z}}
\newcommand{\RN}{\mathbb{R}}

\newcommand{\CutTimes}{\mathcal{A}}
\newcommand{\Cluster}{A}
\newcommand{\Charact}{\tilde{C}}

\newcommand{\func}{g}

\title{Scaling limit of the collision measure for two-dimensional random walks}

\author{Shuta Nakajima \thanks{Department of Mathematics, Keio University. E-mail: njima@keio.jp}
\and Ryoichiro Noda\thanks{Faculty of Science and Engineering, Waseda University.
E-mail: r\textunderscore noda@aoni.waseda.jp}}
\date{}
\begin{document}
\maketitle

\begin{abstract}
We study the scaling limit of the collision measure of two i.i.d. discrete-time random walks on $\ZN^2$, which records their collision sites and times.
This is a critical regime for collisions: the walks collide infinitely often, whereas the limiting Brownian motions do not collide at positive times.
Hence, existing general results on the convergence of collision measures do not apply. 
Assuming that the jump distribution has mean zero and finite second moment, we prove that, under a logarithmic scaling, the collision measure converges in distribution to a non-trivial random measure.
We also give an explicit representation of the limiting measure in terms of a Poisson point process.
To prove convergence, we introduce a new cluster decomposition of the Laplace transform of the collision measure, revealing the cluster structure of collisions.
\end{abstract}

\section{Introduction}

Given two stochastic processes on a common state space, we say that they collide if they visit a common point at the same time.
Collisions are, in themselves, a natural object in the study of stochastic processes.
Moreover, in the polymer model, collisions naturally appear in moment formulae for the associated partition function, and their asymptotics play a crucial role in the analysis of the model; cf.\ \cite{Lygkonis_Zygouras_23_Moments}.

The study of collisions can be traced back to P\'{o}lya's work~\cite{Polya_1921_recurrence}
regarding the recurrence and transience of simple random walks on the lattice $\mathbb{Z}^d$.
Indeed, in this setting,
the recurrence of a walk is equivalent to infinite collisions of a walk and an independent copy.
In general, this equivalence fails~\cite{Krishnapur_Peres_04_Recurrent},
and the problem of such infinite collisions of random walks has been investigated for various graphs;
see, e.g., \cite{Barlow_Peres_Sousi_12_Collisions,DeAmbroggio_Nitzschner_Scali_2025_Collisions,Halberstam_Hutchcroft_22_Collision,Watanabe_23_ICPfor3dUST}.
Another direction concerns the fractal geometry of collisions,
such as Hausdorff dimensions of the sets of collision times and collision sites.
This line of study was initiated by Jain and Pruitt for stable processes~\cite{Jain_Pruitt_69_Collisions},
and has been extended to more general stochastic processes; see~\cite{Knopova_Schilling_15_CollisionofFeller,Noda_pre_Exact,Shiozawa_Wang_24_Hausdorff}.

In this paper, we study scaling limits of collision measures,
that is, space-time point processes consisting of collision sites and collision times of two independent random walks.
In \cite{Nguyen_23_Collision}, 
Nguyen considered the case of collisions of independent simple random walks on $\mathbb{Z}$.
He showed that the associated collision measure converges,
under the usual Brownian scaling,
to a non-trivial space-time random measure.
The idea of his approach is to connect the collision measure with the partition function of a spatially inhomogeneous directed polymer model through certain moment identities.
Using a different approach based on the theory of positive continuous additive functionals of Markov processes, the second author \cite{Noda_pre_STOM} established a general framework for obtaining scaling limits of collision measures, applicable to random walks on various (random) graphs.

We consider such a problem for two i.i.d.\ random walks on $\mathbb{Z}^2$ whose jump distribution has mean zero and finite second moment.
This lies in a critical regime in the sense that they collide infinitely often,
while the limiting processes, i.e., two independent two-dimensional Brownian motions, never collide at positive times.
In particular, this setting is outside the framework of \cite{Noda_pre_STOM}.
To see this critical behavior more quantitatively, consider two independent simple random walks on $\mathbb{Z}^2$ as a concrete example.
Let $N_T$ be their total number of collisions up to time $T$.
Then $N_T$ grows logarithmically in $T$; more precisely, the classical theorem of Erd\H{o}s and Taylor~\cite{Erdos_Taylor_60_Some} gives
\begin{equation}
    \frac{\pi}{\log T}N_T \xrightarrow[T\to\infty]{\mathrm{d}}\mathrm{Exp}(1),
\end{equation}
where $\mathrm{Exp}(1)$ stands for the exponential distribution with mean $1$ and the arrow $\xrightarrow{\mathrm{d}}$ denotes convergence in distribution.
Extensions of the Erd\H{o}s--Taylor theorem to the joint asymptotics of the numbers of pairwise collisions for multiple simple random walks on $\ZN^2$ have been studied in \cite{Cosco_Nakajima_25_High,Cosco_Zeitouni_23_MomentsI,Liu_26_Independence,Lygkonis_Zygouras_24_Multivariate}. 
However, the joint scaling limit of collision sites and times has remained open even for two walks.

Returning to the general class of random walks described above, we show in Theorem~\ref{thm:main} below that the collision measure of the two walks converges, under a suitable logarithmic scaling, to a non-trivial random measure. 
Moreover, in Theorem~\ref{thm:limit meas}, we give an explicit representation of the limiting measure in terms of a Poisson point process.
These results reveal the structure of collisions: collision times form clusters, and the spatial locations of distinct clusters become asymptotically independent of one another and of the collision times.
Since collision measures naturally arise in the analysis of the partition functions of spatially inhomogeneous directed polymers (see \cite[Section~2]{Nguyen_23_Collision}), we expect the results and methods developed here to be useful in the study of such models.

To state the main results, we introduce some notation.
We write $\ZN$ and $\RN$ for the sets of integers and real numbers, respectively.
We set $\NN \coloneqq \ZN \cap (0,\infty)$.
Let $S=(S_n)_{n\geq0}$ be a random walk on $\ZN^2$, 
i.e., a discrete-time Markov chain such that $(S_{n+1}-S_n)_{n\geq0}$ is an i.i.d.\ sequence of $\ZN^2$-valued random vectors.
We assume that the jump distribution, i.e., the law of $S_1 - S_0$, has mean zero and a finite second moment.
Moreover, we assume that $S$ is irreducible, that is, for any $x \in \ZN^2$, we have $P(S_n = x) > 0$ for some $n$, where $P$ denotes the underlying probability measure for $S$ started from the origin.

Let $S^1 = (S^1_n)_{n \geq 0}$ and $S^2 = (S^2_n)_{n \geq 0}$ be independent copies of $S$.
We are interested in a random measure on $\RN^2 \times [0, \infty)$ given by  
\begin{equation}
    \Pi(dx\, dt) 
    \coloneqq 
    \sum_{n=1}^\infty \delta_{(S^1_n, n)}(dx\, dt)\, 
    \mathbf{1}_{\{S^1_n = S^2_n\}},
\end{equation}
which we call the \emph{collision measure} of $S^1$ and $S^2$.
Here, $\mathbf 1_{\cdot}$ denotes the indicator function
and $\delta_{(x,t)}$ denotes the Dirac measure at $(x,t)$. 
To consider its scaling limit,
we introduce a family $(\phi_N)_{N \geq 2}$ of space--time scaling maps from $\ZN^2 \times \NN$ to $\RN^2 \times [0,\infty)$ as follows:
\begin{equation} \label{eq:phi}
    \phi_N(x,t) 
    \coloneqq 
    \left( \frac{x}{\sqrt{t}}, \frac{\log t}{\log N}\right).
\end{equation}
Using this, for each $N \geq 2$, we define the scaled collision measure $\Pi_N$ by 
\begin{equation} \label{eq:Pi_N}
    \Pi_N(dx\, dt) 
    \coloneqq 
    \frac{1}{\log N} 
    \sum_{n = 1}^\infty \delta_{\phi_N(S^1_n, n)}(dx\, dt) 
    \mathbf{1}_{\{S^1_n = S^2_n\}}.
\end{equation}
In the following theorem, we establish the distributional convergence of $\Pi_N$ in the vague topology.

\begin{thm} \label{thm:main}
    Run both $S^1$ and $S^2$ from the origin.
    Then the scaled collision measure $\Pi_N$ converges to a non-trivial random measure $\Pi_*$ in distribution with respect to the vague topology, that is, 
    for any compactly supported continuous function $f\colon \RN^2 \times [0,\infty) \to \RN$,
    \begin{align}
        \int_{\RN^2 \times [0,\infty)} f(x, t)\, \Pi_N(dx\, dt)
        &=
        \frac{1}{\log N} 
        \sum_{n = 1}^\infty 
        f\!\left( \frac{S^1_n}{\sqrt{n}}, \frac{\log n}{\log N} \right)
        \mathbf{1}_{\{S^1_n = S^2_n\}}\\[0.2em]
        &\xrightarrow[N \to \infty]{\mathrm{d}}
        \int_{\RN^2 \times [0,\infty)} f(x, t)\, \Pi_*(dx\, dt).
    \end{align}
\end{thm}

The limiting random measure can be represented by a Poisson point process.
Let $d$ be the periodicity of $S$,
i.e., the greatest common divisor of the set $\{n \geq 1 \mid P(S_n = (0,0)) > 0\}$.
Let $\Gamma$ denote the covariance matrix of $S_1-S_0$,
and set  
\begin{equation}\label{eq:Charact}  
    \Charact 
    \coloneqq 
    \frac{4\pi\sqrt{\det\Gamma}}{d}.
\end{equation}
Let $\nu_{\Gamma/2}$ be the two-dimensional Gaussian distribution
with mean $0$ and covariance matrix $\Gamma/2$.

In the following theorem, we identify the distribution of the limiting measure $\Pi_*$.

\begin{thm} \label{thm:limit meas}
    Let $M(dt\,dx\,dw)$ be a Poisson random measure on
    $(0,\infty) \times \RN^2 \times [0,\infty)$ with intensity measure
    \begin{equation}
        \Charact t^{-2} e^{-\Charact w/t}\,dt\,\nu_{\Gamma/2}(dx)\,dw,
    \end{equation}
    and write $\{(t_i,x_i,w_i)\}_{i\in I}$ for its atoms. 
    Then the limiting random measure $\Pi_*$ in Theorem~\ref{thm:main} 
    has the same distribution as $\sum_{i\in I} w_i \delta_{(x_i, t_i)}$.
\end{thm}

The limiting random measure admits a natural interpretation.
Its temporal component can be understood through renewal theory.
Let $T_0 \coloneqq 0$ and $T_{n+1} \coloneqq \inf\{m > T_n \mid S^1_m = S^2_m\}$.
By the strong Markov property, the successive gaps $(T_{n+1} - T_n)_{n \geq 0}$ between collision times are i.i.d.\ random variables with tail $P(T_1 > u)\sim\Charact/\log u$; see Lemma~\ref{lem:col time tail}.
Thus, standard limit theory for renewal processes with slowly varying tails (cf.\ \cite[Theorem~2.1]{Kabluchko_Marynych_16_Renewal} and \cite[Proposition~4.8(iii)]{Resnick_08_Extrem}) yields
\begin{equation}
    \frac{1}{\log N}
    \sum_{n = 1}^\infty \mathbf{1}_{\{S^1_n=S^2_n\}}
    \delta_{\frac{\log n}{\log N}}
    \xrightarrow[N\to\infty]{\mathrm{d}}
    \sum_{i\in I}w_i\delta_{t_i}
\end{equation}
in the vague topology.
The law of each spatial mark $x_i$ also has a natural explanation.
As shown in Proposition~\ref{prop:IP on col} below, the local central limit theorem implies that the conditional law of $S^1_n/\sqrt{n}$ given $S^1_n=S^2_n$ converges weakly to $\nu_{\Gamma/2}$.
Theorem~\ref{thm:limit meas} shows that the full limiting measure $\Pi_*$ is obtained by assigning to the atoms $t_i$ i.i.d.\ spatial marks $x_i$ with this law, independently of the times $t_i$ and weights $w_i$.

This suggests the following picture: collision times form clusters, with $t_i$ representing the limiting rescaled starting time of a cluster and $w_i$ the limit of the normalized number of collisions within that cluster.
Under the space--time scaling $\phi_N$, the collision sites and times within each cluster collapse to a single point $(x_i,t_i)$.
Moreover, the spatial locations of distinct clusters are asymptotically independent of one another and also of the collision times.

The main difficulty is to prove this asymptotic independence rigorously.
To overcome it, we introduce a new cluster decomposition of the Laplace transform of the collision measure, inspired by chaos expansions for directed polymer models; see Section~\ref{sec:clust decomp}.

The rest of the paper is organized as follows.
Section~\ref{sec:preliminary} collects preliminary estimates for the random walks.
In Sections~\ref{sec:sum integral} and~\ref{sec:limit num col}, we identify the limiting spatial distribution conditioned on a collision and the scaling limit of the number of collisions, respectively.
Section~\ref{sec:conv of Lap} develops the cluster decomposition and uses it to prove convergence of the Laplace transforms.
In Section~\ref{sec:proof}, we complete the proofs of Theorems~\ref{thm:main} and~\ref{thm:limit meas}.
Appendix~\ref{sec:Riemann sum} provides a logarithmic Riemann-sum approximation used in the proof.

\medskip
\textbf{Notation and conventions.}
\begin{enumerate}[leftmargin=*, label=\textup{(\arabic*)}]
    \item
    Unless otherwise stated, all random walks start from the origin,
    and their law and expectation are denoted by $P$ and $E$.
    For a random walk started at $x$, we write $P_x$ and $E_x$.

    \item
    For a measure $\mu$ and a measurable function $f$ for which the
    integral is well defined, we write
    $\mu(f)\coloneqq\int f\,d\mu$.

    \item For a real-valued function $f$ on a set $T$,
    we write $\|f\|_\infty \coloneqq \sup\{|f(x)| \mid x \in T\}$.

    \item 
    For $x = (x_1,\ldots,x_n) \in \RN^n$,
    we denote its Euclidean norm by $|x| = \sqrt{x^2_1+\cdots+x^2_n}$.

    \item
    For $a,b\in\RN$, we write $a\vee b\coloneqq\max\{a,b\}$
    and $a\wedge b\coloneqq\min\{a,b\}$.

    \item
    We write $A\lesssim B$ if $A\leq CB$ for a constant $C\geq1$
    depending only on the jump distribution of $S$.
    We write $A\gtrsim B$ if $B\lesssim A$, and $A\asymp B$ if both
    $A\lesssim B$ and $A\gtrsim B$.
    Additional dependence on parameters $q_1,\ldots,q_k$ is indicated
    by subscripts, as in $\lesssim_{q_1,\ldots,q_k}$.

    \item
    We use parameter-dependent Landau notation.
    The notation $O_{q_1,\ldots,q_k}(1)$ denotes a quantity bounded,
    uniformly in the variables under consideration, by a finite constant
    depending only on parameters $q_1,\ldots,q_k$ and the jump distribution of $S$.
    Similarly, $o_{q_1,\ldots,q_k}(1;x\to x_0)$ denotes a quantity
    converging to zero as $x\to x_0$ with $q_1,\ldots,q_k$ fixed.
    A parameter is omitted from the subscript when the bound or convergence
    is uniform in that parameter.
    For iterated limits, the order is listed from the inner limit to the
    outer one. Thus,
    \begin{equation}
        r_{l,m,n}
        =
        o(1;l\to\infty\, \triangleright\, m\to\infty\, \triangleright\, n\to\infty)
    \end{equation}
    means that
    \begin{equation}
        \limsup_{n\to\infty}
        \limsup_{m\to\infty}
        \limsup_{l\to\infty}|r_{l,m,n}|=0.
    \end{equation}
\end{enumerate}

\section{Preliminary results}
\label{sec:preliminary}

In this section,
we present preliminary results
that are used in the proofs of our main theorems.
We recall the local central limit theorem (LCLT) in Section~\ref{sec:LCLT}.
In Sections~\ref{sec:number col} and~\ref{sec:approx indep},
we establish several consequences of the LCLT.

\subsection{The local central limit theorem}
\label{sec:LCLT}

To state the local central limit theorem,
we introduce several pieces of notation.
Define the positive definite quadratic form $Q\colon \RN^2 \to [0,\infty)$ by setting $Q(x) \coloneqq \langle x, \Gamma^{-1}x \rangle$,
where $\langle \cdot,\cdot \rangle$ denotes the usual inner product on $\RN^2$.
We will often use the following elementary comparison:
\begin{equation} \label{eq:Q asymp}
    Q(x) \asymp |x|^2,\qquad \forall x \in \RN^2
\end{equation}
(cf. \cite[Proposition~1.1.1(c)]{Lawler_Limic_10_Random}).
Since the random walk has period $d$,
we have a natural decomposition $\ZN^2=\bigsqcup_{i=0}^{d-1}\ZN_i^2$,
where, for each $i \in \{0,\ldots,d-1\}$,
\begin{equation}\label{def: ZNi}
    \ZN_i^2
    \coloneqq
    \bigl\{x \in \ZN^2 : P(S_m=x)>0
    \text{ for some }m \in \ZN_{\geq 0}\text{ such that }m\equiv i \pmod d\bigr\}.
\end{equation}
For convenience,
for each $n \in \ZN_{\geq 0}$, we set
$\ZN_n^2\coloneqq\ZN_i^2$ whenever $n\equiv i\pmod d$.

The following standard lemma allows us to transfer results
from the aperiodic case $d=1$ to the periodic case $d>1$
(e.g. \cite[Proof of Equation~(5.9)]{Dembo_Peres_Rosen_Zeitouni_01_Thick}).
For completeness, we include a brief proof.

\begin{lem} \label{lem:lattice transform}
    There exists an invertible $2\times2$ matrix $B$ such that
    $\ZN^2_0 = B \ZN^2$ and $|\det B| = d$.
\end{lem}

\begin{proof}
    The irreducibility implies that $\ZN^2_0$ is a $2$-dimensional sublattice of $\ZN^2$, that is, $\ZN^2_0=\{mv_1+nv_2 \mid m,n\in\ZN\}$ for some linearly independent vectors $v_1,v_2\in\ZN^2$.
    By \cite[Theorem~1 in Chapter~I]{Cassels_97_Introduction},
    there exists a $2\times2$ matrix $B$ such that $\ZN^2_0 = B\ZN^2$.
    Since the index of $\ZN^2_0$ as a subgroup of $\ZN^2$ is $d$
    (see \cite[Remark following P1 in Section~5]{Spitzer_76_Principles}),
    it follows from \cite[Lemma~1 in Chapter~I]{Cassels_97_Introduction} that $|\det B| = d$.
    In particular, $B$ is invertible, which completes the proof.
\end{proof}

We use the following form of the local central limit theorem (LCLT).
Define 
\begin{equation} \label{eq:c_*}
    c_* 
    \coloneqq \frac{d}{2\pi \sqrt{\det \Gamma}}.
\end{equation}

\begin{lem} [{LCLT}]\label{lem:LCLT}
    There exists a non-increasing sequence $(\rho_n)_{n \geq 1}$ of positive numbers that converges to $0$ as $n\to \infty$ such that,
    for any $n \geq 1$,
    \begin{equation} \label{eq:LCLT 1}
        \left|
            P(S_n = x)
            -
            \frac{c_{*}}{n} e^{-Q(x)/(2n)}
        \right|
        \leq
        \frac{\rho_n}{n \vee |x|^{2}},
        \qquad 
        \forall x \in \ZN^2_n
    \end{equation}
    and, in particular,
    \begin{equation} \label{eq:LCLT 2}
        P(S_n = x) \lesssim \frac{1}{n \vee |x|^2},
        \qquad 
        \forall x \in \ZN^2.
    \end{equation}
\end{lem}

{
\begin{proof}
    When the jump distribution is aperiodic,
    \eqref{eq:LCLT 1} follows from \cite[Theorems~2.3.9 and~2.3.10]{Lawler_Limic_10_Random}.
    The periodic case follows from the standard reduction to the aperiodic case, as in \cite[Proof of Theorem~2.1.3]{Lawler_Limic_10_Random}.
    Indeed, 
    let $B$ be the $2\times2$ matrix as in Lemma~\ref{lem:lattice transform}.
    Then the definition of $\ZN^2_0$ implies that $(B^{-1}S_{dk})_{k \geq 0}$ is an aperiodic random walk on $\ZN^2$.
    Thus, the result at times that are multiples of $d$ follows from the aperiodic case and Lemma~\ref{lem:lattice transform}.
    For the remaining residue classes, the Markov property gives, for any $x \in \ZN^2$ and $k \geq 1$,
    \begin{equation} \label{eq:Markov decomp}
        P(S_{dk+i}=x)
        =
        \sum_{y\in\ZN^2}
        P(S_{dk}=x-y)P(S_i=y),
        \qquad \forall i\in\{0,\ldots,d-1\}.
    \end{equation}
    Write $p_s(z)\coloneqq s^{-1}e^{-Q(z)/(2s)}$ for $s>0$ and $z\in\RN^2$.
    The aperiodic estimate gives
    \begin{equation} \label{eq:aperiodic LCLT}
        |P(S_{dk}=z)-c_* p_{dk}(z)|\leq\frac{\tilde{\rho}_k}{dk\vee|z|^2},
        \qquad z\in\ZN_0^2
    \end{equation}
    for some decreasing sequence $(\tilde{\rho}_k)_{k \geq 1}$ of positive numbers converging to $0$ as $k\to \infty$.
    Recall that $Q(z) =\langle z, \Gamma^{-1}z \rangle$.
    Since $\nabla Q(z) = 2\Gamma^{-1}z$,
    it follows that 
    \begin{equation}
        \nabla p_s(z)
         =
        -\frac{1}{2}s^{-2}e^{-Q(z)/(2s)} \cdot \nabla Q(z)
        =
        -s^{-2} e^{-Q(z)/(2s)} \cdot \Gamma^{-1}z.
    \end{equation}
    Using that $Q(z) \gtrsim |z|^2$ from \eqref{eq:Q asymp}, $|\Gamma^{-1}z| \lesssim |z|$,  and $re^{-r^2} \lesssim  e^{-r^2/2}$ for all $r \geq 0$ ,
    we obtain, for some $c > 0$,
    \begin{equation} \label{eq:spatial deriv est}
        |\nabla p_s(z)|
        \lesssim s^{-3/2}
        e^{-c |z|^2/s},
        \qquad \forall s>0,\, z \in \RN^2.
    \end{equation}
    Similarly, we have 
    \begin{equation} 
    \partial_s p_s(z)
    =
    s^{-2}
    \left(\frac{Q(z)}{2s}-1\right)
    e^{-Q(z)/(2s)}.
    \end{equation}
    Using $Q(z)\gtrsim  |z|^2$ again and $r e^{-r} \lesssim e^{-r/2}$ for all $r \geq 0$,
    we obtain, for some $c' > 0$,
    \begin{equation} \label{eq:temporal deriv est}
        |\partial_s p_s(z)|
        \lesssim 
        s^{-2}
        e^{-c' |z|^2/s},
        \qquad \forall s>0,\, z \in \RN^2.
    \end{equation}
    These bounds imply the following continuity estimates for all
    $s>0$, $x,y\in\RN^2$, and $0\leq h\leq s$:
    \begin{align}
        (s\vee|x|^2)|p_s(x-y)-p_s(x)|
        &\lesssim \frac{|y|}{\sqrt{s}}+\frac{|y|^2}{s},
        \label{eq:spatial conti}\\
        (s\vee|x|^2)|p_{s+h}(x)-p_s(x)|
        &\lesssim \frac{h}{s}.
        \label{eq:temporal conti}
    \end{align}
    Indeed, if $|y|\leq\frac12\sqrt{s\vee|x|^2}$, the mean value theorem and \eqref{eq:spatial deriv est} yield
    \begin{equation}
        (s\vee|x|^2)|p_s(x-y)-p_s(x)|
        \leq |y|(s\vee|x|^2)
        \sup_{0\leq\theta\leq1}|\nabla p_s(x-\theta y)|
        \lesssim\frac{|y|}{\sqrt{s}},
    \end{equation}
    where we use $\inf_{0 \leq \theta \leq 1} |x-\theta y|\geq|x|/2$ when $|x|>\sqrt{s}$
    and use $\sup_z |\nabla p_s(z)| \lesssim s^{-3/2}$ when $|x| \leq \sqrt{s}$.
    On the other hand, if $|y| > \frac12\sqrt{s\vee|x|^2}$, then the bound $\sup_z p_s(z)\lesssim s^{-1}$ gives $(s\vee|x|^2)|p_s(x-y)-p_s(x)|\lesssim |y|^2/s$.
    Hence, we obtain \eqref{eq:spatial conti}.
    The temporal estimate~\eqref{eq:temporal conti} follows from a similar argument using the mean value theorem and \eqref{eq:temporal deriv est}.

    For any $x,y \in \ZN^2$ and $k\geq1$, since $|x|^2 \leq 2|x-y|^2 + 2|y|^2$, we have
    \begin{equation}
        \frac{dk\vee|x|^2}{dk\vee|x-y|^2}
        \leq 2+\frac{2|y|^2}{dk}.
    \end{equation}
    Combining this with \eqref{eq:aperiodic LCLT}, \eqref{eq:spatial conti}, \eqref{eq:temporal conti}, and the triangle inequality, we obtain, for any $k \geq 1$, $i \in \{0,\ldots,d-1\}$, $x \in \ZN^2_{dk+i}$, and $y \in \ZN^2_i$,
    \begin{equation}
        (dk\vee|x|^2)|P(S_{dk}=x-y)-c_*p_{dk+i}(x)|
        \lesssim \tilde{\rho}_k\left(1+\frac{|y|^2}{dk}\right)
        +\frac{|y|}{\sqrt{dk}}+\frac{|y|^2+i}{dk}.
    \end{equation}
    From this and \eqref{eq:Markov decomp}, it follows that
    \begin{equation}
        (dk\vee|x|^2)|P(S_{dk+i}=x) - c_* p_{dk+i}(x)|
        \lesssim \tilde{\rho}_k\left(1+\frac{E[|S_i|^2]}{dk}\right) + \frac{E[|S_i|]}{\sqrt{dk}} + \frac{E[|S_i|^2]+i}{dk}.
    \end{equation}
    Let $\rho'_k$ denote the supremum of the right-hand side over $i\in\{0,\ldots,d-1\}$.
    Since $S_i$ has a finite second moment for each $i$, we have that $\rho'_k \to 0$ as $k \to \infty$.
    In particular, if we set $\rho''_k \coloneqq \sup_{m \geq k} \rho'_m$, then $\rho''_k$ is non-increasing, converges to $0$ as $k \to \infty$, and satisfies
    \begin{equation}
        |P(S_{dk+i}=x)-c_*p_{dk+i}(x)|
        \lesssim \frac{\rho''_k}{dk\vee|x|^2},
        \qquad \forall x \in \ZN^2_{dk+i},\, i \in \{0,\ldots,d-1\}.  
    \end{equation}
    Since $(dk+i)\vee|x|^2\asymp dk\vee|x|^2$, by setting $\rho_{dk+i} \coloneqq C\rho''_k$ for a sufficiently large constant $C$, we obtain \eqref{eq:LCLT 1} for $n\geq d$.
    For $1\leq n<d$, the estimate \eqref{eq:LCLT 1} holds by taking $\rho_n$ sufficiently large.
    This completes the proof. 
\end{proof}
}

\subsection{Moment estimates for the number of collisions}
\label{sec:number col}

Using the LCLT, we give an upper bound for the moment of the number of collisions in Lemma~\ref{lem:col est} below.
Recall that $S^1$ and $S^2$ are independent copies of $S$.

\begin{lem} \label{lem:col est}
    There exists a constant $c_1 \in (0,\infty)$ such that,
    for any integers $0 = t_0 < t_1 < \cdots < t_m < \infty$,
    \begin{equation}\label{eq:col est 1}
        \sup_{x,y \in \ZN^2}
        E_{(x,y)}\!
        \left[
            \prod_{i=1}^m
            \mathbf{1}_{\{S^1_{t_i} = S^2_{t_i}\}}
        \right]
        \leq     
        c_1^m
        \prod_{i=1}^m 
        \frac{1}{t_i - t_{i-1}}.
    \end{equation}
    As a consequence, 
    there exists a constant $c_2 \in (0,\infty)$ such that, 
    for any integers $N \geq 2$ and $m \leq N$,
    \begin{equation} \label{eq:col est 2}
        \sum_{0 < t_1 < \cdots < t_m \leq N} 
        \sup_{x,y \in \ZN^2}
        E_{(x,y)}\!
        \left[
            \prod_{i=1}^m
            \mathbf{1}_{\{S^1_{t_i} = S^2_{t_i}\}}
        \right]
        \leq    
        (c_2 \log N)^m.
    \end{equation}
\end{lem}

\begin{proof}
    By independence, translation invariance, and Lemma~\ref{lem:LCLT},
    for any $n \geq 1$ and $x,y \in \ZN^2$,
    \begin{align}
        P_{(x,y)}(S^1_n = S^2_n)
        &=
        \sum_{z \in \ZN^2} P(S_n = z-x)P(S_n = z-y)\\
        &\leq
        \sup_{u \in \ZN^2} P(S_n = u)\\
        &\lesssim \frac{1}{n}.
    \end{align}
    Using the Markov property and applying the above bound successively, we obtain \eqref{eq:col est 1}.
    The second result~\eqref{eq:col est 2} follows  from \eqref{eq:col est 1} 
    combined with
    $\sum_{t=1}^N t^{-1} \lesssim \log N$ and 
    \begin{equation} 
        \sum_{0 = t_0 < t_1 < \cdots < t_m \leq N} 
        \prod_{i=1}^m 
        \frac{1}{t_i - t_{i-1}}    
        \leq 
        \left(
            \sum_{t=1}^N \frac{1}{t}
        \right)^m
        \leq   
        (1+\log N)^m.
    \end{equation}
    This completes the proof.
\end{proof}

\subsection{Approximate independence at separated times}
\label{sec:approx indep}

At well-separated times, a random walk should look almost independent of its past. 
In this subsection, we give a quantitative estimate of this heuristic for collisions of the two walks;
see Proposition~\ref{prop:asymp indep} below.

Define $\bm{S}_n \coloneqq (S^1_n, S^2_n)$ for $n \geq 0$.
Then $\bm{S} = (\bm{S}_n)_{n \geq 0}$ is a random walk on $\ZN^4$. Define
\begin{equation} \label{eq: dfn of D}
    D \coloneqq \{(x, x) \in \ZN^4 \mid x \in \ZN^2\}.
\end{equation}
Note that a collision of $S^{1}$ and $S^{2}$ is equivalent to a visit to $D$ by $\bm{S}$.

\begin{prop} \label{prop:asymp indep}
Recall the sequence  $(\rho_N)_{N \geq 1}$ from Lemma~\ref{lem:LCLT}.    Let $r_0 \in (1, \infty)$.  
    For any $M, N \in \NN$ with $M > 2N$ and any bounded function $f\colon \ZN^4 \to \RN$ supported on $D$,
    almost surely on the event $\{|\bm{S}_N| \leq r_0\sqrt{N}\}$,
    \begin{equation}
        \left|
            E\!\left[ f(\bm{S}_M)\, \mathbf 1_{\{|\bm{S}_M| \leq r_0 \sqrt{M}\}}\, |\, \bm{S}_N \right]
            -
            E\!\left[ f(\bm{S}_M)\, \mathbf 1_{\{|\bm{S}_M| \leq r_0\sqrt{M}\}}\right]
        \right|
        \lesssim_{r_0}
        \frac{ \|f\|_\infty}{M}
        \left(
            \sqrt{\frac{N}{M}} + \rho_N
        \right).
        \end{equation}
\end{prop}

To prove the above proposition,
we present two lemmas.

\begin{lem}\label{lem:HK diff}
    For all $r_0 > 0$, $M, N \in \NN$, and $x,y\in\mathbb{Z}^2$ satisfying
    $M> 2N$ and $|x|\leq r_0 \sqrt{N}$, one has
    \begin{equation}
    \left|
    (M-N)^{-1}e^{-Q(y-x)/(2(M-N))}
    -
    M^{-1}e^{-Q(y)/(2M)}
    \right|
    \lesssim_{r_0}
    M^{-1}
    \sqrt{\frac{N}{M}}.
    \end{equation}
\end{lem}

\begin{proof}
    Recall $p_s(z) = s^{-1}e^{-Q(z)/(2s)}$ from the proof of Lemma~\ref{lem:LCLT}.
    The triangle inequality yields
    \begin{align}
    \left|p_{ M-N}(y-x)-p_M(y)\right|
    &\leq
    \left|p_{ M-N}(y-x)-p_{ M-N}(y)\right|
    +
    \left|p_{ M-N}(y)-p_M(y)\right|.
    \label{eq:HK triangle}
    \end{align}
    By \eqref{eq:spatial deriv est} and the mean value theorem,
    \begin{align}
        \left|p_{ M-N}(y-x)-p_{ M-N}(y)\right|
        &\lesssim
        |x|(M-N)^{-3/2}
        \sup_{0\leq \theta\leq 1}
        \exp\!\left(-c \frac{|y-\theta x|^2}{ M-N}\right)\\
        &\leq     
        |x| (M-N)^{-3/2}.
    \end{align}
    Using $ M-N > M/2$ and $|x|\leq r_0\sqrt{N}$, we obtain
    \begin{equation}\label{eq: spacial diff}
        \left|p_{ M-N}(y-x)-p_{ M-N}(y)\right|
        \lesssim_{r_0}
        \frac{1}{M} \sqrt{\frac{N}{M}}.
    \end{equation}
    By \eqref{eq:temporal deriv est} and the mean value theorem, 
    \begin{equation}
        \left|p_{ M-N}(y)-p_M(y)\right|
        \lesssim
        N 
        \sup_{M-N \leq s \leq M}
        s^{-2}
        e^{-c' |y|^2/s}
        \leq    
        N (M-N)^{-2}.
    \end{equation}
    Using $M-N \geq M/2$ again, we obtain
    \begin{equation}\label{eq: time diff}
        \left|p_{M-N}(y)-p_M(y)\right|
        \lesssim
        N M^{-2}
        \leq
        \frac{1}{M} \sqrt{\frac{N}{M}},
    \end{equation}
    where the last inequality follows from $N/M\leq \sqrt{N/M}$.
    Combining \eqref{eq:HK triangle} with \eqref{eq: spacial diff} and \eqref{eq: time diff} completes the proof.
\end{proof}

\begin{lem} \label{lem:RW HK diff}
    For all $r_0>0$ and $N, M \in \NN$ with $M > 2N$,
    if $x \in \ZN^2_N$ with $|x|\leq r_0\sqrt N$ and $y \in \ZN^2_M$,
    then
    \begin{equation}
        \left| P(S_{M-N} = y - x) - P(S_M = y) \right|
        \lesssim_{r_0}
        M^{-1}
        \sqrt{\frac{N}{M}}
        + 
        \frac{\rho_N}{M}.
    \end{equation}
\end{lem}

\begin{proof}
    Since $x\in\ZN_N^2$ and $y\in\ZN_M^2$, 
    we have $y-x\in\ZN_{M-N}^2$.
    Thus, from the LCLT, Lemma~\ref{lem:HK diff}, and the triangle inequality,
    it follows that
    \begin{equation}
        \left| P(S_{M-N} = y - x) - P(S_M = y) \right|
        \lesssim_{r_0}
        \frac{1}{M}
        \sqrt{\frac{N}{M}}
        + 
         \frac{\rho_{M-N}}{M-N}
        +
        \frac{\rho_M}{M}.
    \end{equation}
    Using $M > 2N$ and $\rho_{M-N} \leq \rho_N$, we obtain the desired result.
\end{proof}

We are now ready to prove Proposition~\ref{prop:asymp indep}.

\begin{proof}[Proof of Proposition~\ref{prop:asymp indep}]
    Fix $r_0 \in (1, \infty)$, $M, N \in \NN$ with $M > 2N$.
    Let $\bm{x} = (x_1, x_2) \in \ZN^2_N \times \ZN^2_N$ be such that 
    $P(\bm{S}_N = \bm{x}) > 0$ and $|\bm{x}| \leq r_0 \sqrt{N}$.
    Set $A_M \coloneqq \left\{y\in\ZN_M^2:\, |(y,y)|\le r_0\sqrt M\right\}$.
    By the Markov property and independence of \(S^1\) and \(S^2\),
    \begin{align}
        &E\!\left[ f(\bm{S}_M)\mathbf 1_{\{|\bm{S}_M|\le r_0\sqrt M\}}\,\middle|\,\bm{S}_N=(x_1,x_2)\right] \\
        &=
        \sum_{y\in A_M}
        f(y,y)P(S_{M-N}=y-x_1)P(S_{M-N}=y-x_2),
    \end{align}
    whereas
    \begin{equation}
        E\!\left[ f(\bm{S}_M)\mathbf 1_{\{|\bm{S}_M|\le r_0\sqrt M\}} \right]
        =
        \sum_{y\in A_M}f(y,y)P(S_M=y)^2.
    \end{equation}
    Since $M > 2N$,
    the LCLT implies that
    \begin{equation}   
        \max_{n\geq M-N,\,x\in \ZN^2}P(S_{n}=x) \lesssim \frac{1}{M}.
    \end{equation}
    Hence, 
    using \(|ab-c^2|\le b|a-c|+c|b-c|\) for $a,b,c \geq 0$,
    we deduce that
    \begin{align}
        &\left|
            E\!\left[ f(\bm{S}_M)\mathbf 1_{\{|\bm{S}_M|\le r_0\sqrt M\}} \,\middle|\,\bm{S}_N = (x_1, x_2) \right]
            -
            E\!\left[ f(\bm{S}_M)\mathbf 1_{\{|\bm{S}_M|\le r_0\sqrt M\}} \right]
        \right|\\
         &= \left| \sum_{y\in A_M}f(y,y)\bigl( P(S_{M-N}=y-x_1)P(S_{M-N}=y-x_2)-P(S_M=y)^2 \bigr) \right|\\
        &\lesssim 
        \frac{\|f\|_\infty}{M}
        \sum_{i=1}^2
        \sum_{y\in A_M}
        \left|P(S_{M-N}=y-x_i)-P(S_M=y)\right|. 
        \label{eq:proof-main}
    \end{align}
    Lemma~\ref{lem:RW HK diff} then yields that, for each $i=1,2$,
    \begin{equation}
        \sum_{y\in A_M}
        \left|P(S_{M-N}=y-x_i)-P(S_M=y)\right|
        \lesssim_{r_0}
        M^{-1}\sqrt{\frac NM} \sum_{y \in A_M} 1
        +
        \frac{\rho_N}{M}\sum_{y \in A_M} 1.
    \end{equation}
    Since $\sum_{y \in A_M} 1 \lesssim_{r_0} M$,
    the desired result follows.
\end{proof}

\section{Scaling limit of the spatial distribution at collision times} \label{sec:sum integral}

In this section,
using the LCLT,
we identify the limiting spatial distribution of a walk conditioned on a collision.

\begin{prop} \label{prop:IP on col}
    Let $f$ be a bounded continuous function on $\RN^2$.
    Then 
    \begin{equation}
        E\bigl[ f\bigl( S^1_n/\sqrt{n} \bigr) \mid S^1_n = S^2_n \bigr]
        = 
        \int_{\RN^2} f(x)\, \nu_{\Gamma/2}(dx) +   o_f(1;n\to \infty),
    \end{equation}
    where we recall 
    that $\nu_{\Gamma/2}$ denotes the two-dimensional Gaussian distribution with mean $0$ and covariance matrix $\Gamma/2$. 
\end{prop}

\begin{proof}
    By the independence of $S^1$ and $S^2$,
    we have 
    \begin{equation}
        E\bigl[ f\bigl( S^1_n/\sqrt{n} \bigr) \mid S^1_n = S^2_n \bigr]
        =
        P(S^1_n = S^2_n)^{-1}
        \sum_{x \in \ZN^2_n}
        f(x/\sqrt{n}) P(S_n = x)^2.
    \end{equation}
    Since the covariance matrix of the jump distribution of  $\tilde{S} \coloneqq S^1 - S^2$ is $2 \Gamma$,
    applying the LCLT to $\tilde{S}$ gives
    \begin{equation} \label{eq:IP on col1}
        P(S^1_n = S^2_n) 
        = 
        P(\tilde{S}_n = (0,0))
        =
        \frac{c_*}{2n}
        +
        \frac{o(1; n \to \infty)}{n}.
    \end{equation}
    Combining this with the LCLT for $S$,
    we deduce that 
    \begin{align}
        E\bigl[ f\bigl( S^1_n/\sqrt{n} \bigr) \mid S^1_n = S^2_n \bigr]
        &=
        2 \sum_{x \in \ZN^2_n} f(x/\sqrt{n}) e^{-Q(x)/(2n)} P(S_n = x)
        + (\|f\|_\infty+1) o(1;n \to \infty)\\
        &=
        2 E\bigl[ f\bigl( S_n/\sqrt{n} \bigr) e^{-Q(S_n)/(2n)} \bigr]
        + o_f(1;n\to\infty).
    \end{align}
    By the central limit theorem for $S_n/\sqrt{n}$,
    we obtain 
    \begin{equation}
        \lim_{n \to \infty}
        E\bigl[ f\bigl( S^1_n/\sqrt{n} \bigr) \mid S^1_n = S^2_n \bigr]
        =
        \frac{1}{\pi\sqrt{\det \Gamma}}
        \int_{\RN^2} f(x) e^{-Q(x)}\, dx.
    \end{equation}
    Finally, the definition of $\nu_{\Gamma/2}$ implies that 
    the right-hand side coincides with $\int_{\RN^2} f(x)\, \nu_{\Gamma/2}(dx)$,
    which completes the proof.
\end{proof}

For later use,
we record a consequence of Proposition~\ref{prop:IP on col}.

\begin{cor} \label{cor:sum to int}
    Let $\{f_\alpha\}_{\alpha \in \mathcal{A}}$ be a family of functions from $\RN^2$ to $\RN$.
    Assume that $\{f_\alpha\}_{\alpha \in \mathcal{A}}$ is uniformly bounded and equicontinuous.
    Then 
    \begin{equation}
        \sup_{\alpha \in \mathcal{A}}
        \left|
            E\!\left[ f_\alpha( S^1_n/\sqrt{n}) \mathbf{1}_{\{S^1_n = S^2_n\}} \right] 
            -
            \frac{c_*}{2n} 
             \int_{\RN^2} f_\alpha(x)\, \nu_{\Gamma/2}(dx)
        \right|
        =
        \frac{o_{\{f_\alpha\}_{\alpha \in \mathcal{A}}}(1; n \to \infty)}{n}.
    \end{equation}
\end{cor}

\begin{proof}
    By Proposition~\ref{prop:IP on col} and~\cite[Theorem~3.1]{RangaRao_62_Relations},
    we have 
    \begin{equation}
        \sup_{\alpha \in \mathcal{A}}
        \left|
            E\bigl[ f_\alpha\bigl( S^1_n/\sqrt{n} \bigr) \mid S^1_n = S^2_n \bigr]
            -
             \int_{\RN^2} f_\alpha(x)\, \nu_{\Gamma/2}(dx)
        \right|
        = o_{\{f_\alpha\}_{\alpha \in \mathcal{A}}}(1; n \to \infty).
    \end{equation}
    Combining this with \eqref{eq:IP on col1} and the uniform boundedness of $\{f_\alpha\}_{\alpha \in \mathcal{A}}$,
    we obtain the desired result.
\end{proof}

\section{Scaling limit of the number of collisions}
\label{sec:limit num col}

In this section, we show that the number of collisions converges
under logarithmic scaling to an exponential random variable;
see Theorem~\ref{thm:col num conv}.
The number of collisions is the local time at the origin
of the difference walk $\tilde{S}\coloneqq S^1-S^2$.
Such scaling limits go back to the Erd{\H{o}}s--Taylor theorem
\cite[Theorem~1]{Erdos_Taylor_60_Some};
see also \cite{Bohun_Marynych_21_Local,Lygkonis_Zygouras_23_Moments}
for generalizations.
Since $\tilde{S}$ may be confined to a proper sublattice of $\ZN^2$,
applying the existing results directly requires some care.
We therefore include a short proof adapted to our setting. 
Recall the constant $\tilde{C}$ from \eqref{eq:Charact}.

\begin{thm} \label{thm:col num conv}
    For each $\lambda \geq 0$,
    \begin{equation}
        \lim_{T \to \infty}
        E\!\left[
            \exp\!\left( -\frac{\lambda}{\log T} \sum_{n=1}^{T} \mathbf{1}\{S^1_n = S^2_n\} \right)
        \right]
        = 
        \frac{\Charact}{\lambda + \Charact}.
    \end{equation}
    Moreover, this convergence is uniform in $\lambda$ on every compact subset of $[0,\infty)$.
\end{thm}

For later use, we record the following version of the above theorem.

\begin{cor} \label{cor:col num conv}
    The following convergence is uniform for $(\lambda,t)$ 
    in every compact subset of $[0,\infty)\times(0,\infty)$:
    \begin{equation}
        \lim_{N \to \infty}
        E\!\left[
            \exp\!\left( -\frac{\lambda}{\log N} \sum_{n=1}^{\lfloor N^t \rfloor} \mathbf{1}_{\{S^1_n=S^2_n\}} \right)
        \right]
        =
        \frac{\Charact}{\lambda\, t + \Charact}.
    \end{equation}
\end{cor}

\begin{proof}
    Applying Theorem~\ref{thm:col num conv} with $\lfloor N^t\rfloor$ in place of $T$ 
    and using $\log\lfloor N^t\rfloor/\log N\to t$,
    one can readily verify the result. 
\end{proof}

To establish Theorem~\ref{thm:col num conv},
we first identify the tail asymptotic of the first collision time $\xi$,
defined by
\begin{equation}
    \xi \coloneqq \inf\{n \geq 1 \mid S^1_n = S^2_n\}.
\end{equation}
Below, given functions $f$ and $g$,
we write $f(x) \sim g(x)$ as $x \to \infty$
if $f(x)/g(x) \to 1$ as $x \to \infty$.

\begin{lem} \label{lem:col time tail}
    As $x\to\infty$,
    \begin{equation}
        P(\xi > x) \sim \frac{\Charact}{\log x}.
    \end{equation}
\end{lem}

\begin{proof}
Recall the notation $\ZN_i$ from \eqref{def: ZNi}.     Define $\tilde{S} = (\tilde{S}_n)_{n \geq 0}$
    by $\tilde{S}_n \coloneqq S^1_n - S^2_n$.
    We then have $\xi = \inf\{n \geq 1 \mid \tilde{S}_n = (0,0) \}$.
    The subgroup generated by the support of $\tilde{S}_1$ is $\ZN_0^2$.
    Moreover, $\tilde{S}$ is irreducible and aperiodic on $\ZN_0^2$.
    Let $B$ be the matrix in Lemma~\ref{lem:lattice transform}.
    Then $S'\coloneqq B^{-1}\tilde{S}$ is an irreducible and
    aperiodic random walk on $\ZN^2$.
    Let $\tilde{\Gamma}$ and $\Gamma'$ denote the covariance matrices of $\tilde{S}$ and $S'$,
    respectively.
    The independence of $S^1$ and $S^2$ implies $\tilde{\Gamma} = 2\Gamma$.
    Moreover, since $|\det B| = d$,
    it follows that $\det \Gamma' = d^{-2} \det \tilde{\Gamma}$.
    Thus, $\det \Gamma' = 4d^{-2} \det \Gamma$.
    Applying \cite[Proposition~4.2.4]{Lawler_Limic_10_Random} to $S'$,
    we obtain 
    \begin{equation}
        P(\xi > n)
        \sim
        \frac{2\pi\sqrt{\det\Gamma'}}{\log n}
        =
        \frac{4\pi\sqrt{\det\Gamma}}{d\log n}
        =
        \frac{\Charact}{\log n}.
    \end{equation}
    Finally, since $P(\xi>x)=P(\xi>\lfloor x\rfloor)$, the same
    asymptotic holds for real $x$.
\end{proof}

\begin{proof} [Proof of Theorem~\ref{thm:col num conv}]
    Put $T_0\coloneqq0$ and
    $T_{n+1}\coloneqq\inf\{m>T_n \mid S^1_m=S^2_m\}$ for $n\geq0$, 
    and set $\xi_{n+1}\coloneqq T_{n+1}-T_n$.  
    Then $\xi_1=\xi$ and $(\xi_n)_{n\geq1}$ is an i.i.d.\ sequence.
    Define the corresponding first-passage process by
    $\mathcal{N}(t) \coloneqq \inf\{n\geq1 \mid T_n>t\}$ for $t \geq 0$.
    By definition,
    $\sum_{n=1}^T \mathbf{1}_{\{S^1_n=S^2_n\}} = \mathcal{N}(T)-1$.
    Set $L(x)\coloneqq\Charact^{-1}\log x$ for $x \geq 1$.
    By Lemma~\ref{lem:col time tail},
    $P(\xi > x) \sim L(x)^{-1}$ as $x \to \infty$.
    Hence, by \cite[Corollary~2.3]{Kabluchko_Marynych_16_Renewal}, as $T\to \infty$, 
    $\mathcal{N}(T)/L(T)$ converges in distribution to the standard exponential distribution.
    Recalling that $L(T) = \Charact^{-1} \log T$,
    we obtain the convergence of the Laplace transforms.
    It is well known that 
    the convergence is uniform on every compact subset 
    (cf.\ \cite[Theorem~6.3(ii)]{Kallenberg_21_Foundations}).
    This completes the proof.
\end{proof}

\section{Convergence of Laplace transforms} \label{sec:conv of Lap}

Distributional convergence of random measures can be verified by convergence of relevant Laplace transforms~\cite{Kallenberg_17_Random}.
The aim of this section is to prove the convergence of Laplace transforms for the scaled collision measures $\Pi_N$;
see Theorem~\ref{thm:lap conv} below.

We impose the following conditions on the test functions.

\begin{assum}[The class of test functions] \label{assum:test}
    The test function $f\colon \RN^2 \times [0,\infty) \to [0,\infty)$ 
    is uniformly continuous
    and satisfies the following condition:
    there exists $\kappa = \kappa(f)\in (1,\infty)$ 
    such that $\supp(f) \subseteq \RN^2 \times [0,\kappa-1]$ and
    \begin{equation} \label{eq:norm of f}
        c_2 \kappa \|f\|_\infty < 1,
    \end{equation}
    where $\supp(f)$ denotes the support of $f$ and 
    the constant $c_2$ is recalled from Lemma~\ref{lem:col est}.
\end{assum}

Throughout this section,
we fix a test function $f$ satisfying Assumption~\ref{assum:test}.
The following is the main theorem of this section.  

\begin{thm}\label{thm:lap conv}
    We have
    \begin{align}
        \lim_{N \to \infty}
        E\!\left[e^{-\Pi_N(f)}\right]
        = 
        \exp\!\left(
            -
            \int_0^\infty \int_{\RN^2} 
            \frac{f(x, t)}{f(x, t)\,t + \Charact}\, \nu_{\Gamma/2}(dx)\,dt
        \right).
        \label{eq:conv of lap}
    \end{align}
\end{thm}

We prove the above theorem through a sequence of approximation steps developed in the following subsections.

\subsection{Approximation by an early-time cutoff} 
\label{sec:cutoff early}

In this subsection,
we introduce an early-time cutoff of the collision measure 
and show that the resulting Laplace transform approximates that of $\Pi_N$;
see Proposition~\ref{prop:early col}.
  
For each $\delta \in (0,1)$ and $N \geq 2$,
we define a random measure $\Pi_N^{(\delta)}$ on $\RN^2 \times [0,\infty)$ by 
\begin{equation} \label{eq:tilde Pi truncated}
    \Pi_N^{(\delta)}(dx\, dt)
    \coloneqq 
    \frac{1}{\log N} 
    \sum_{n = \lceil N^{\delta} \rceil}^\infty  
    \delta_{\phi_N(S^1_n, n)}(dx\, dt)\, 
    \mathbf{1}_D(\bm{S}_n),
\end{equation}
where we recall the scaling map $\phi_N$ from \eqref{eq:phi} and
$D=\{(x,x)\in\ZN^4\mid x\in\ZN^2\}$.
Compared to $\Pi_N$,
the random measure $\Pi_N^{(\delta)}$ ignores early collisions before time $N^\delta$.
The following proposition shows that the error caused by this cutoff
vanishes as $\delta\downarrow0$ after $N\to\infty$.

\begin{prop}
\label{prop:early col}
    We have
    \begin{align}
        \limsup_{\delta \to 0}
        \limsup_{N \to \infty}
        \left|
            E\!\left[e^{-\Pi_N^{(\delta)}(f)}\right]
            -
            E\!\left[e^{-\Pi_N(f)}\right]
        \right|
        = 0.
        \label{eq:conv of lap truncated}
    \end{align}
\end{prop}

\begin{proof}
    By the mean value theorem,
    we have 
    \begin{equation} \label{eq:e^-x conti}
        |e^{-x} - e^{-y}| \leq |x-y|, \qquad \forall x,y \geq 0.
    \end{equation}
    This yields 
    \begin{align}
         \left|
            E\!\left[e^{-\Pi_N^{(\delta)}(f)}\right]
            -
            E\!\left[e^{-\Pi_N(f)}\right]
        \right|
    &\leq 
    E\bigl[ |\Pi_N^{(\delta)}(f)- \Pi_N(f)| \bigr]\\
    &\leq 
    \frac{\|f\|_\infty}{\log N} 
    \sum_{n =1}^{ \lceil N^{\delta} \rceil}  
    P(S^1_n = S^2_n)\\
    &\lesssim 
    \frac{\|f\|_\infty}{\log N} 
    \log \lceil N^{\delta} \rceil,
    \end{align}
    where we used Lemma~\ref{lem:col est} to obtain the last inequality.
    Since $\log\lceil N^\delta\rceil/\log N \to \delta$ as $N \to \infty$, 
    we obtain the desired result.
\end{proof}

\subsection{Cluster decomposition}
\label{sec:clust decomp}

By Proposition~\ref{prop:early col},
it remains to identify, for each fixed small $\delta>0$,
the limit of the Laplace transform of $\Pi_N^{(\delta)}(f)$ as $N\to\infty$.
To this end, we rewrite the Laplace transform by grouping nearby collision times into clusters.
For each $N\geq2$, $x \in \ZN^2$, and $t \in \NN$,
set 
\begin{equation}
    f_N(x, t) 
    \coloneqq 
    \frac{1}{\log N}\,f\bigl( \phi_N(x, t) \bigr),
    \qquad
    \func_N(x, t) 
    \coloneqq 
    e^{-f_N(x, t)} - 1.
\end{equation}
For each $\varepsilon \in (0,1/2)$, $N \geq 2$, $a \in \NN$, 
and $\bm{X}=(X^1, X^2)$ with $X^1,X^2\colon \NN \cup \{0\} \to \ZN^2$,
we write
\begin{gather}
    h_{N, a}^{(\varepsilon)}(\bm{X}) 
    \coloneqq 
    \exp
    \left(
        {-} \sum_{k=1}^{\lfloor a^{1-\varepsilon} \rfloor} 
        f_N(X^1_k, k+a)
        \mathbf{1}_D\!\left(\bm{X}_k\right)
    \right),
\end{gather}
where we denote $\bm{X}_k\coloneqq (X_k^1,X_k^2)$. 
We also define, for each $N\geq2$, $L \geq 1$, $\delta\in(0,1)$, and $u\in(-1,1)$,
\begin{align} 
    \CutTimes_L(u)
    &\coloneqq
    \CutTimes_L(u,\delta,N)\\
    &\coloneqq 
    \{(a_1, \ldots, a_L) \in [N^\delta, N^{\kappa }]^L \cap \ZN^L
        \mid 
        a_l > a_{l-1} + \lfloor a_{l-1}^{1+u} \rfloor\quad  
        \forall  l \geq 2
    \}.
    \label{eq:Cut time}
\end{align}

The following is the cluster decomposition.

\begin{prop} 
\label{prop:chaos}
    Write $\bm{S}_\cdot \coloneqq (S^1_\cdot, S^2_\cdot)$ and consider it as a random map from $\NN \cup \{0\}$ to $\RN^4$.
    Let $(\theta_n)_{n \geq 0}$ be the time-shift operator of $(\bm{S}_n)_{n \geq 1}$, i.e., a family of maps satisfying $\bm{S}_k \circ \theta_n = \bm{S}_{k+n}$ for each $n,k \geq 0$. 
    Then, for each $N\geq2$, $\delta\in(0,1)$, and $\varepsilon\in(0,1/2)$,
    \begin{align}
        E\!\left[e^{-\Pi_N^{(\delta)}(f)}\right]
        &= 
        1
        + 
        \sum_{L=1}^{\infty}
        \sum_{(a_l)_{l=1}^L \in \CutTimes_L(-\varepsilon)}\!
        E\!
        \left[
            \prod_{l=1}^{L} 
            \func_N(S^1_{a_l}, a_l) 
            \mathbf{1}_D(\bm{S}_{a_l})
            h^{(\varepsilon)}_{N,a_l}(\bm{S} \circ \theta_{a_l})
        \right].
    \end{align}
\end{prop}

Before proving Proposition~\ref{prop:chaos},
we explain the interpretation of the decomposition above.
Each ordered collection of times appearing in the expansion is partitioned into clusters, 
where $L$ is the number of clusters and $a_l$ is the starting time of the $l$-th cluster.
The sum over $(a_l)_{l=1}^L$ therefore runs over all possible choices of cluster starting times.
For each $l$, 
the factor $\func_N(S^1_{a_l},a_l)\mathbf{1}_D(\bm{S}_{a_l})$
represents the contribution at the beginning of the cluster, while
$h_{N,a_l}^{(\varepsilon)}$ collects the contributions from the subsequent collisions within that cluster. 
Thus, 
the proposition reorganizes the Laplace transform according to the cluster structure of collision times.

\begin{proof} [Proof of Proposition~\ref{prop:chaos}]
    By the definition of $\Pi_N^{(\delta)}$
    and the support assumption on $f$, 
    we have 
    \begin{equation} \label{eq:chaos1}
        E\!\left[e^{-\Pi_N^{(\delta)}(f)}\right]
        = 
        E\!\left[
            \exp\!\left(
                -\sum_{n = \lceil N^{\delta} \rceil}^{\lfloor N^\kappa \rfloor}  
                f_N(S^1_n, n) 
                \mathbf{1}_D(\bm{S}_n)
            \right)
        \right].
    \end{equation}
    We then apply the following expansion to the right-hand side: 
    for a function $Q\colon [0,\infty) \to \RN$ and a finite subset $I$ of $[0,\infty)$,
    \begin{equation} \label{eq:simple chaos}
        e^{\sum_{t \in I} Q(t)} = \prod_{t \in I} \left( 1+
(e^{ Q(t)} -1)     \right)  = 1 + \sum_{m=1}^\infty 
        \sum_{\substack{t_1,\ldots, t_m \in I\\ t_1 < \cdots < t_m}}
        \prod_{i=1}^m \left(e^{Q(t_i)} - 1\right).
    \end{equation} 
    We apply this identity pointwise with $Q(t)=-f_N(S^1_t,t)\mathbf{1}_D(\bm{S}_t)$. 
    Since $e^{Q(t)}-1=\func_N(S^1_t,t)\mathbf{1}_D(\bm{S}_t)$, we have
\begin{equation} \label{eq:chaos2}
        E\!\left[e^{-\Pi_N^{(\delta)}(f)}\right]
        = 
        1 + 
        \sum_{m=1}^\infty    
        \sum_{\lceil N^\delta \rceil \leq t_1 < \cdots < t_m \leq \lfloor N^{\kappa} \rfloor}
        E\!\left[
            \prod_{i=1}^m
            \func_N(S^1_{t_i}, t_i) 
            \mathbf{1}_D(\bm{S}_{t_i})
        \right].
    \end{equation}
    We set, for each $a \in \NN$ and $k \in \NN \cup \{0\}$,
    \begin{equation} \label{eq:chaos3}
        \Cluster_k^a(-\varepsilon) 
        \coloneqq 
        \{(s_0, \ldots, s_k) \in \ZN^{k+1} 
            \mid     
            a = s_0 < s_1 < \ldots < s_k \leq a + \lfloor a^{1-\varepsilon} \rfloor
        \}.
    \end{equation}
   Given $\lceil N^\delta \rceil\leq t_1 < \cdots < t_m\leq \lfloor N^{\kappa}\rfloor$, we uniquely associate $(a_l)_{l=1}^L \in \CutTimes_L(-\varepsilon)$ and for each $l$, $(s_{l,j})_{j=0}^{k_l} \in \Cluster_{k_l}^{a_l}(-\varepsilon)$ such that $(t_i)_{i=1}^m = ((s_{l,j})_{j=0}^{k_l})_{l=1}^L$. Following this correspondence, by rearranging the sum on the right-hand side of \eqref{eq:chaos2}, 
    we obtain
    \begin{align} \label{eq: want to de-coupling 1}
        E\!\left[e^{-\Pi_N^{(\delta)}(f)}\right]
        &= 
        1
        + 
        \sum_{L=1}^{\infty}
        \sum_{(a_l)_{l=1}^L \in \CutTimes_L(-\varepsilon)}
        \sum_{\substack{0 \leq k_l < \infty\\ l = 1, \ldots, L}}
        \sum_{(s_{l,j})_{j=0}^{k_l} \in \Cluster_{k_l}^{a_l}(-\varepsilon)}
        \hspace{-2mm}
        E\!\left[
            \prod_{l=1}^{L} 
            \prod_{j=0}^{k_l}
            \func_N(S^1_{s_{l, j}}, s_{l, j}) 
            \mathbf{1}_D(\bm{S}_{s_{l,j}})
        \right].
    \end{align}
    Using the expansion \eqref{eq:simple chaos} again, 
    we deduce that, for each $(a_l)_{l=1}^L \in \CutTimes_L(-\varepsilon)$,
    \begin{align}
        &\sum_{\substack{0 \leq k_l < \infty\\ l = 1, \ldots, L}}
        \sum_{(s_{l,j})_{j=0}^{k_l} \in \Cluster_{k_l}^{a_l}(-\varepsilon)}
        E\!\left[
            \prod_{l=1}^{L} 
            \prod_{j=0}^{k_l}
            \func_N(S^1_{s_{l, j}}, s_{l, j}) 
            \mathbf{1}_D(\bm{S}_{s_{l,j}})
        \right]\\
        &=
        E\!\left[
            \prod_{l=1}^{L} 
            \sum_{k_l = 0}^\infty
            \sum_{(s_{l,j})_{j=0}^{k_l} \in \Cluster_{k_l}^{a_l}(-\varepsilon)}
            \prod_{j=0}^{k_l}
            \func_N(S^1_{s_{l, j}}, s_{l, j}) 
            \mathbf{1}_D(\bm{S}_{s_{l,j}})
        \right]\\
        &= 
        E\!\left[
            \prod_{l=1}^{L} 
            \func_N(S^1_{a_l}, a_l) 
            \mathbf{1}_D(\bm{S}_{a_l})
            \exp\biggl(
                - \sum_{k=a_l+1}^{a_l + \lfloor a_l^{1-\varepsilon} \rfloor} 
                f_N(S^1_k, k) 
                \mathbf{1}_D(\bm{S}_k)
            \biggr)
        \right].
    \end{align}
    Moreover, using the shift operator, we have
    \begin{equation}
        \exp\biggl(
            - \sum_{k=a_l+1}^{a_l + \lfloor a_l^{1-\varepsilon} \rfloor} 
            f_N(S^1_k, k) 
            \mathbf{1}_D(\bm{S}_k)
        \biggr)
        =
        h^{(\varepsilon)}_{N,a_l}(\bm{S} \circ \theta_{a_l}),
    \end{equation}
    which completes the proof.
\end{proof}

\subsection{Negligible cluster configurations}
\label{sec:neg config}
In this subsection, we show that configurations with excessively many clusters, clusters that are too close to each other, or clusters lying outside a diffusive annulus give negligible contributions to the cluster expansion; see Propositions~\ref{prop:many clust}, \ref{prop:small gaps}, and~\ref{prop:annulus}. 

In the arguments below, 
we often use the following elementary bounds: 
\begin{equation} \label{eq:bound f,g,h}
    \|f_N\|_\infty \vee \|\func_N\|_\infty 
    \leq 
    \frac{\|f\|_{\infty}}{\log N},
    \qquad
    \|h_{N, a}^{(\varepsilon)}\|_\infty \leq 1.
\end{equation}

We first show that the cluster expansion can be truncated at a finite number
of clusters, uniformly in $N$, $\delta$, and $\varepsilon$.

\begin{prop} [Ignoring excessively many clusters]
    \label{prop:many clust}
    For each $N \geq 2$, $\varepsilon \in (0,1/2)$, and $\delta \in (0,1)$,
    \begin{align}
        \sum_{L=M+1}^{\infty}
        \sum_{(a_l)_{l=1}^L \in \CutTimes_L(-\varepsilon)}
        E\!\left[
            \prod_{l=1}^{L} 
            \func_N(S^1_{a_l}, a_l) 
            \mathbf{1}_D(\bm{S}_{a_l})
            h_{N,a_l}^{(\varepsilon)}(\bm{S} \circ \theta_{a_l})
        \right]
        = o_f(1;M \to \infty).
    \end{align}
\end{prop}

\begin{proof}
    By \eqref{eq:bound f,g,h}, 
    we deduce that for any $N\geq 2$,
    \begin{align}
        &\Biggl|
        \sum_{L=M+1}^{\infty}
        \sum_{(a_l)_{l=1}^L \in \CutTimes_L(-\varepsilon)}
        E\!
        \left[
            \prod_{l=1}^{L} 
            \func_N(S^1_{a_l}, a_l) 
            \mathbf{1}_D(\bm{S}_{a_l})
            h_{N,a_l}^{(\varepsilon)}(\bm{S} \circ \theta_{a_l})
        \right]
        \Biggr|
        \\
        &\leq
        \sum_{L=M+1}^{\infty}
        \Bigl(\frac{\|f\|_\infty}{\log N}\Bigr)^{L}
        \sum_{(a_l)_{l=1}^L \in \CutTimes_L(-\varepsilon)}
        E\!
        \left[
            \prod_{l=1}^{L} 
            \mathbf{1}_D(\bm{S}_{a_l})
        \right]\\
        &\leq
        \sum_{L=M+1}^{\infty}
        (c_2 \kappa \|f\|_\infty)^L
        \\
        &\lesssim_f  (c_2 \kappa \|f\|_\infty)^{M+1},
    \end{align}
    where we used \eqref{eq:col est 2} and \eqref{eq:norm of f} to obtain the second and third inequalities, respectively. 
    Thus, the desired result follows.
\end{proof}

Recall that, in the cluster decomposition, $(a_l)_{l=1}^L$ are the cluster-starting times. 
We next show that the sum may be restricted to configurations with sufficiently well-separated clusters, 
namely, those satisfying $a_l-a_{l-1}>\lfloor a_{l-1}^{1+\varepsilon}\rfloor$ for every $l=2,\ldots,L$.

\begin{prop} [Ignoring close clusters] \label{prop:small gaps}
    Recall the set $\CutTimes_L(\varepsilon)$ from \eqref{eq:Cut time}:
    \begin{equation}
        \CutTimes_L(\varepsilon)
        \coloneqq 
        \{(a_1, \ldots, a_L) \in [N^\delta, N^{\kappa }]^L \cap \ZN^L
            \mid 
            a_l > a_{l-1} + \lfloor a_{l-1}^{1+\varepsilon} \rfloor\quad \forall l \geq 2
        \}.
    \end{equation}
    Then, for each $\delta \in (0,1)$, $\varepsilon \in (0,1/2)$, and $L \geq 1$,
    \begin{align}
        &\sum_{(a_l)_{l=1}^L \in \CutTimes_L(-\varepsilon)}
        E\!\left[
            \prod_{l=1}^{L} 
            \func_N(S^1_{a_l}, a_l) 
            \mathbf{1}_D(\bm{S}_{a_l})
            h_{N,a_l}^{(\varepsilon)}(\bm{S} \circ \theta_{a_l})
        \right]
        \\
        &=
        \sum_{(a_l)_{l=1}^L \in \CutTimes_L({\varepsilon})}
        E\!\left[
            \prod_{l=1}^{L} 
            \func_N(S^1_{a_l}, a_l) 
            \mathbf{1}_D(\bm{S}_{a_l})
            h_{N,a_l}^{(\varepsilon)}(\bm{S} \circ \theta_{a_l})
        \right]
        +o_{f,\delta,L}(1;N \to \infty\, \triangleright\, \varepsilon \to 0).
    \end{align}
\end{prop}

\begin{proof}
    For $L=1$, 
    we have $\CutTimes_1(\varepsilon)=\CutTimes_1(-\varepsilon)$, 
    so the claim is immediate.  
    Hence, we assume $L\geq2$ below.
    By definition, $\CutTimes_L(\varepsilon) \subseteq \CutTimes_L(-\varepsilon)$, 
    and the difference between the two sets consists of $(a_l)_{l=1}^{L} \in \CutTimes_L(-\varepsilon)$ such that there exists an $i \in \{2,\ldots, L\}$ satisfying $a_i \leq a_{i-1} + \lfloor a_{i-1}^{1+\varepsilon} \rfloor$.
    Thus, if we write, for each $(a_l)_{l = 1}^L \in \CutTimes_L(-\varepsilon)$,
    \begin{equation}
        H_N((a_l)_{l=1}^L) \coloneqq
        \prod_{l=1}^{L} 
            \func_N(S^1_{a_l}, a_l) 
            \mathbf{1}_D(\bm{S}_{a_l})
            h^{(\varepsilon)}_{N,a_l}(\bm{S} \circ \theta_{a_l}),
    \end{equation}
    then 
    \begin{equation}
        \left|
            \sum_{(a_l)_{l=1}^L \in \CutTimes_L(-\varepsilon)}\hspace{-1em}
            E\bigl[ H_N((a_l)_{l=1}^L) \bigr]
            -\hspace{-0.8em}
            \sum_{(a_l)_{l=1}^L \in \CutTimes_L(\varepsilon)}\hspace{-1em}
            E\bigl[ H_N((a_l)_{l=1}^L) \bigr]
        \right|
        \leq   
        \sum_{i=2}^{L}
        \sum_{\substack{(a_l)_{l=1}^{L} \in \CutTimes_L(-\varepsilon)\\
        a_i \leq a_{i-1} + \lfloor a_{i-1}^{1+\varepsilon} \rfloor}}\!
        E\bigl[ |H_N((a_l)_{l=1}^L)| \bigr].
    \end{equation}
    Hence, using \eqref{eq:bound f,g,h} and a trivial bound $e^{-x} \leq 1$ for any $x \geq 0$,
    the right-hand side of the above inequality is bounded by 
    \begin{align}
        &
        \sum_{i=2}^{L}
        \sum_{\substack{(a_l)_{l=1}^{L} \in \CutTimes_L(-\varepsilon)\\
        a_i \leq a_{i-1} + \lfloor a_{i-1}^{1+\varepsilon} \rfloor}}\!
        \Bigl(\frac{\|f\|_\infty}{\log N}\Bigr)^{L}
        E\!
        \left[
            \prod_{l=1}^{L}
            \mathbf{1}_D(\bm{S}_{a_l})
        \right]
        \leq
        \sum_{i=2}^{L}
        \sum_{\substack{(a_l)_{l=1}^{L} \in \CutTimes_L(-\varepsilon)\\
        a_i \leq a_{i-1} + \lfloor a_{i-1}^{1+\varepsilon} \rfloor}}\!
        \Bigl(\frac{c_1\|f\|_\infty}{\log N}\Bigr)^L
        \prod_{l=1}^{L}
        \frac{1}{a_l-a_{l-1}},
        \label{eq:bound on atypically small gaps}
    \end{align}
    where we used Lemma~\ref{lem:col est} to obtain the inequality and set $a_0=0$. 
    We now bound the second sum on the right-hand side of the above inequality.
    Fix $i \in \{2,\ldots,L\}$.
    For each fixed natural number $a_{i-1} \in [N^\delta, N^\kappa]$,
    we have
    \begin{equation}
        \frac{1}{\log N}
        \sum_{a_i = a_{i-1} + \lfloor a_{i-1}^{1-\varepsilon} \rfloor}^{a_{i-1} + \lfloor a_{i-1}^{1+\varepsilon} \rfloor}
        \frac{1}{a_i - a_{i-1}}
        \leq
        \frac{1}{\log N}
        \log\!\left(\frac{\lfloor a_{i-1}^{1+\varepsilon} \rfloor}{\lfloor a_{i-1}^{1-\varepsilon} \rfloor-1}\right)
        =
        o_{f,\delta}(1;N \to \infty\, \triangleright\, \varepsilon \to 0).
    \end{equation}
    Hence,
    for each $i \in \{2,\ldots,L\}$,
    we deduce that 
    \begin{align}
        \frac{1}{(\log N)^L}
        \sum_{\substack{(a_l)_{l=1}^{L} \in \CutTimes_L(-\varepsilon)\\
        a_i \leq a_{i-1} + \lfloor a_{i-1}^{1+\varepsilon} \rfloor}}
        \prod_{l=1}^{L}
        \frac{1}{a_l-a_{l-1}}
        &\leq  
        \frac{o_{f,\delta}(1;N \to \infty\, \triangleright\, \varepsilon \to 0)}{(\log N)^{L-1}}
        \left( \sum_{m = 2}^{\lfloor N^\kappa \rfloor} \frac{1}{m}  \right)^{L-1}\\
        &=
        o_{f,\delta,L}(1;N \to \infty\, \triangleright\, \varepsilon \to 0).
    \end{align}
    This, combined with~\eqref{eq:bound on atypically small gaps}, yields the desired result.
\end{proof}

We finally show that the cluster decomposition may be restricted, up to a negligible error,
to configurations in which the collision at each cluster-starting time $a_l$
occurs at a distance of order $\sqrt{a_l}$ from the origin.
To formulate this restriction, we introduce the following notation.
For each $R > 1$, define
\begin{equation}
    D_R
    \coloneqq 
    \{(x,x) \in \RN^2 \times \RN^2 \mid R^{-1} \leq |x| \leq R\}.
\end{equation}
For each $n \geq 1$,
we write 
\begin{equation}
    \bm{S}^{\dagger}_n \coloneqq n^{-1/2}(S^1_n, S^2_n).
\end{equation}
Note that 
\begin{equation}
    \bm{S}^{\dagger}_n \in D_R
    \quad \Longleftrightarrow \quad   
    S^1_n = S^2_n\quad \text{and}\quad |S^1_n| \in [R^{-1}\sqrt{n}, R\sqrt{n}].
\end{equation}

\begin{prop} [Annulus restriction for clusters]
    \label{prop:annulus}
    For each $L \geq 1$, $\delta \in (0,1)$, and $\varepsilon \in (0,1/2)$,
    \begin{align}
        &\sum_{(a_l)_{l=1}^L \in \CutTimes_L(\varepsilon)}
        E\!
        \left[
            \prod_{l=1}^{L} 
            \func_N(S^1_{a_l}, a_l) 
            \mathbf{1}_D(\bm{S}_{a_l})
            h_{N,a_l}^{(\varepsilon)}(\bm{S} \circ \theta_{a_l})
        \right]
        \\
        &=
        \sum_{(a_l)_{l=1}^L \in \CutTimes_L(\varepsilon)}
        E\!
        \left[  \prod_{l=1}^{L} 
            \func_N(S^1_{a_l}, a_l) 
            \mathbf{1}_{D_R}(\bm{S}^\dagger_{a_l})
            h_{N,a_l}^{(\varepsilon)}(\bm{S} \circ \theta_{a_l})
        \right]
        +
        o_{f,L}(1; N \to \infty\, \triangleright\, R \to \infty).
    \end{align}
\end{prop}

To prove the above proposition,
we prepare two lemmas below.

\begin{lem} \label{lem:prod diff}
    Fix $L \in \NN$.
    For real numbers $a_1,\ldots,a_L,b_1,\ldots,b_L$, 
    \begin{equation}
        \prod_{i=1}^L a_i-\prod_{i=1}^L b_i 
        = 
        \sum_{\ell=1}^L  (a_\ell-b_\ell) \Big(\prod_{i=1}^{\ell-1} a_i\Big)
        \Big( \prod_{i=\ell+1}^{L } b_i\Big),
    \end{equation}
    where the product over the empty set is defined to be $1$.
\end{lem}

\begin{proof}
    This is elementary and thus we omit the proof.
\end{proof}

\begin{lem} \label{lem:annu neg}
    For each $\kappa>0$, $L\in\NN$, $N\geq2$, and $R>1$,
    \begin{align}
        &\frac{1}{(\log N)^L}
        \sum_{1\leq t_1<\cdots<t_L\leq\lfloor N^\kappa\rfloor}\hspace{-1em}
        P\!\left(
            \bm{S}_{t_l}\in D \text{ for all } l\in \{1,\ldots,L\},
            \text{ and }
            \bm{S}^\dagger_{t_L}\notin D_R
        \right)
        \\
        &\lesssim_{\kappa,L}
        \frac{1}{\log N}+\frac{1}{R^2}.
    \end{align}
\end{lem}

\begin{proof}
    Set $t_0=0$ and write $p_n(x)\coloneqq P(S_n=x)$.
    By the independence of $S^1$ and $S^2$ and the Markov property,
    the probability in the summand is equal to
    \begin{equation}
        \sum_{\substack{x_1,\ldots,x_L\in\ZN^2\\
        |x_L|\notin[\sqrt{t_L}/R,R\sqrt{t_L}]}}
        \prod_{l=1}^L
        p_{t_l-t_{l-1}}(x_l-x_{l-1})^2,
    \end{equation}
    where we set $x_0 \coloneqq (0,0)$.
    By Lemma~\ref{lem:LCLT},
    \begin{equation}
        p_n(x)^2
        \lesssim
        \frac{1}{n}p_n(x),
        \qquad n\geq1,\, x\in\ZN^2.
    \end{equation}
    Thus, summing first over $x_1,\ldots,x_{L-1}$ 
    and using the Chapman--Kolmogorov equation,
    we obtain
    \begin{equation}
        P\!\left(
            \bm{S}_{t_l}\in D \text{ for all } l\in \{1,\ldots,L\},
            \text{ and }
            \bm{S}^\dagger_{t_L}\notin D_R
        \right)
        \lesssim_L
        \prod_{l=1}^L\frac{1}{t_l-t_{l-1}}
        \sum_{|x|\notin[\sqrt{t_L}/R,R\sqrt{t_L}]}
        p_{t_L}(x).
    \end{equation}
    For the inner region, Lemma~\ref{lem:LCLT} gives
    \begin{equation}
        \sum_{|x|<\sqrt{t_L}/R} p_{t_L}(x)
        \lesssim
        \frac{1}{t_L} 
        \sum_{|x|<\sqrt{t_L}/R} 1
        \lesssim
        \frac{1}{t_L}+\frac{1}{R^2}.
    \end{equation}
   For the outer region, 
   Chebyshev's inequality with the finite second moment of $S_1$ yields
    \begin{equation}
        \sum_{|x|>R\sqrt{t_L}}p_{t_L}(x)
        = 
        P\bigl( |S_{t_L}| > R\sqrt{t_L} \bigr)
        \lesssim
        \frac{1}{R^2}.
    \end{equation}
    It follows that
    \begin{align}
        &P\!\left(
            \bm{S}_{t_l}\in D \text{ for all } l\in \{1,\ldots,L\},
            \text{ and }
            \bm{S}^\dagger_{t_L}\notin D_R
        \right)
     \\
        &\lesssim_L
        \left(
            \frac{1}{(t_L-t_{L-1})^2}
            +\frac{1}{(t_L-t_{L-1})R^2}
        \right)
        \prod_{l=1}^{L-1}\frac{1}{t_l-t_{l-1}}.
    \end{align}
    Consequently,
    \begin{align}
        &\frac{1}{(\log N)^L}
        \sum_{1\leq t_1<\cdots<t_L\leq\lfloor N^\kappa\rfloor}\hspace{-1em}
        P\!\left(
            \bm{S}_{t_l}\in D \text{ for all } l\in \{1,\ldots,L\},
            \text{ and }
            \bm{S}^\dagger_{t_L}\notin D_R
        \right)
     \\
        &\lesssim_L
        \frac{1}{(\log N)^L}
        \left(\sum_{n=1}^{\lfloor N^\kappa\rfloor}\frac{1}{n}\right)^{L-1}
        \sum_{n=1}^{\lfloor N^\kappa\rfloor}
        \left(\frac{1}{n^2}+\frac{1}{nR^2}\right)\\
        &\lesssim_{\kappa,L}
        \frac{1}{\log N}+\frac{1}{R^2},
    \end{align}
    which completes the proof.
\end{proof}

We now complete the proof of Proposition~\ref{prop:annulus}.

\begin{proof} [Proof of Proposition~\ref{prop:annulus}]
    By \eqref{eq:bound f,g,h} and Lemma~\ref{lem:prod diff},
    it suffices to show that for any $k\in \{1,\ldots,L\}$, 
    \begin{align}
        &(\log N)^{-L}\hspace{-1em}
        \sum_{(a_l)_{l=1}^L \in \CutTimes_L(\varepsilon)}
        E\!
        \left[
            \mathbf{1}_D(\bm{S}_{a_k})
            \mathbf{1}_{D_R^c}(\bm{S}^{\dagger}_{a_k})        
            \prod_{l=1}^{k-1} \mathbf{1}_D(\bm{S}_{a_l})            
            \prod_{l=k+1}^{L} \mathbf{1}_D(\bm{S}_{a_l})
        \right]\\
        &=    
        o_{f,L}(1; N \to \infty\, \triangleright\, R \to \infty).
    \end{align}
    By the Markov property at time $a_k$ and Lemmas~\ref{lem:col est} and~\ref{lem:annu neg},
    the left-hand side of the preceding display is bounded by
    \begin{align}
        &(\log N)^{-L} \hspace{-1em}
        \sum_{(a_l)_{l=1}^L \in \CutTimes_L(\varepsilon)}\hspace{-0.5em}
        E\!
        \left[
            \mathbf{1}_D(\bm{S}_{a_k})
            \mathbf{1}_{D_R^c}(\bm{S}^{\dagger}_{a_k})        
            \prod_{l=1}^{k-1} \mathbf{1}_D(\bm{S}_{a_l})            
        \right]
        \sup_{x \in \ZN^2} 
        E_{(x,x)}\!\left[
            \prod_{l=k+1}^{L} \mathbf{1}_D(\bm{S}_{a_l-a_k})
        \right]\\
        &\lesssim_{f,L}
        (\log N)^{-k} \hspace{-1em}
        \sum_{(a_l)_{l=1}^k \in \CutTimes_k(\varepsilon)}\hspace{-1em}
        P\!\left(
            \bm{S}_{a_l}\in D \text{ for all } l \in \{1,\ldots,k\},
            \text{ and } 
            \bm{S}^\dagger_{a_k}\notin D_R
        \right)\\
        &= 
        o_{f,L}(1; N \to \infty\, \triangleright\, R \to \infty),
    \end{align}
    which completes the proof.
\end{proof}

Combining Propositions~\ref{prop:many clust}--\ref{prop:annulus} yields the following corollary. 

\begin{cor} \label{cor:neg col}
    For each $\delta \in (0,1)$,
    \begin{align}
        &\sum_{L=1}^{\infty}
        \sum_{(a_l)_{l=1}^L \in \CutTimes_L(-\varepsilon)}\!
        E\!
        \left[
            \prod_{l=1}^{L} 
            \func_N(S^1_{a_l}, a_l) 
            \mathbf{1}_D(\bm{S}_{a_l})
            h^{(\varepsilon)}_{N,a_l}(\bm{S} \circ \theta_{a_l})
        \right]\\
        &= 
        \sum_{L=1}^M
        \sum_{(a_l)_{l=1}^L \in \CutTimes_L(\varepsilon)}
        E\!
        \left[  \prod_{l=1}^{L} 
            \func_N(S^1_{a_l}, a_l) 
            \mathbf{1}_{D_R}(\bm{S}^\dagger_{a_l})
            h_{N,a_l}^{(\varepsilon)}(\bm{S} \circ \theta_{a_l})
        \right]\\
        &\qquad
        +
        o_{f,\delta}(1; N \to \infty\, \triangleright\, R \to \infty\, \triangleright\, \varepsilon \to 0\, \triangleright\, M \to \infty).
    \end{align}
\end{cor}

\subsection{Asymptotic independence of well-separated clusters}
\label{sec:asymp indep}

The main result of this subsection is Proposition~\ref{prop:indep decomp} below,
which shows that 
the contributions from well-separated clusters are asymptotically independent.

\begin{prop} [{Asymptotic independence of clusters}] 
\label{prop:indep decomp}
    For each $L \geq 1$, $\delta \in (0,1)$, and $\varepsilon \in (0,1/2)$,
    \begin{align}
        &\sum_{(a_l)_{l=1}^L \in \CutTimes_L(\varepsilon)}
        E\!
        \left[
            \prod_{l=1}^{L} 
            \func_N(S^1_{a_l}, a_l) 
            \mathbf{1}_{D_R}(\bm{S}^\dagger_{a_l})
            h_{N,a_l}^{(\varepsilon)}(\bm{S} \circ \theta_{a_l})
        \right]
        \\
        &=
        \sum_{(a_l)_{l=1}^L \in \CutTimes_L(\varepsilon)}
        \prod_{l=1}^{L} 
        E\!
        \left[
            \func_N(S^1_{a_l}, a_l) 
            \mathbf{1}_{D_R}(\bm{S}^\dagger_{a_l})
            h_{N,a_l}^{(\varepsilon)}(\bm{S} \circ \theta_{a_l})
        \right]\\
        &\qquad
        +
        o_{f,L,\delta,\varepsilon}(1; N \to \infty\, \triangleright\, R \to \infty).
    \end{align}
\end{prop}
 
Henceforth, given $(a_l)_{l=1}^L \in \CutTimes_L(\varepsilon)$, 
we write 
\begin{equation}
    b_{l} = b_{l}(\varepsilon) \coloneqq a_l + \lfloor a_l^{1-\varepsilon} \rfloor,
\end{equation}
which marks the end of the time window defining the $l$-th cluster.
To apply Proposition~\ref{prop:asymp indep}, 
we first restrict the position of the walk
at each endpoint $b_l$ to a ball of diffusive scale.
The following lemma shows that this additional restriction has negligible cost as $R\to\infty$.

\begin{lem}\label{lem:R rstrct1}
    For each $N \geq 2$, $L \geq 1$, $\delta \in (0,1)$, $\varepsilon \in (0,1/2)$,
    \begin{align}
        &\sum_{(a_l)_{l=1}^L \in \CutTimes_L(\varepsilon)}
         E\!
        \left[ 
            \prod_{l=1}^L 
            \func_N(S^1_{a_l}, a_l) 
            \mathbf{1}_{D_R}(\bm{S}^\dagger_{a_l})
            h_{N, a_l}^{(\varepsilon)}(\bm{S} \circ \theta_{a_l})
        \right]\\
        &=\sum_{(a_l)_{l=1}^L \in \CutTimes_L(\varepsilon)}
        E\!
        \left[
            \prod_{l=1}^{L} 
            \func_N(S^1_{a_l}, a_l) 
            \mathbf{1}_{D_R}(\bm{S}^\dagger_{a_l}) 
            h_{N, a_l}^{(\varepsilon)}(\bm{S} \circ \theta_{a_l})    
            \mathbf{1}_{\{|\bm{S}_{b_l}|\leq 4R \sqrt{b_{l}}\}}   
        \right] 
        + 
        o_{f,L}(1;R\to \infty).\notag
    \end{align}
    We also have
    \begin{align}
        & \sum_{(a_l)_{l=1}^L \in \CutTimes_L(\varepsilon)}
        \prod_{l=1}^{L} 
        E\! 
        \left[  \func_N(S^1_{a_l}, a_l)  
            \mathbf{1}_{D_R}(\bm{S}^\dagger_{a_l})
            h_{N, a_l}^{(\varepsilon)}(\bm{S} \circ \theta_{a_l})
        \right]\label{eq: adding restriction at bl2}\\
        &= 
        \sum_{(a_l)_{l=1}^L \in \CutTimes_L(\varepsilon)}
        \prod_{l=1}^{L} 
        E\! 
        \left[
            \func_N(S^1_{a_l}, a_l)  
            \mathbf{1}_{D_R}(\bm{S}^\dagger_{a_l})
            h_{N, a_l}^{(\varepsilon)}(\bm{S} \circ \theta_{a_l})   
            \mathbf{1}_{\{|\bm{S}_{b_l}|\leq 4R \sqrt{b_{l}}\}} 
        \right]  +
        o_{f,L}(1; R \to \infty).
    \end{align}
\end{lem}

\begin{proof}
    On $\{\bm{S}^\dagger_{a_k}\in D_R\}$, the event
    $\{|\bm{S}_{b_k}|>4R\sqrt{b_k}\}$ implies $|\bm{S}_{b_k}-\bm{S}_{a_k}|>2R\sqrt{a_k}$.
    Chebyshev's inequality and the finite second moment of $|\bm{S}_1|$ yield
    \begin{equation}
        \sup_{\bm{x} \in \ZN^4}
        P_{\bm x}\left(
            |\bm{S}_{b_k}-\bm{S}_{a_k}|>2R\sqrt{a_k}
        \right)
        \lesssim \frac{b_k-a_k}{R^2a_k}
        \lesssim \frac{1}{R^2}.
    \end{equation}
    Thus, by the Markov property and Lemma~\ref{lem:col est}, 
    analogously to the proof of Proposition~\ref{prop:annulus},
    we obtain, for each $k \in \{1,\ldots,L\}$,
    \begin{align}
        &\sum_{(a_l)_{l=1}^L\in\CutTimes_L(\varepsilon)}
        E\!
        \left[
          \left(  \prod_{l=1}^L\mathbf 1_{D_R}(\bm{S}^\dagger_{a_l})\right)
            \mathbf 1_{\{|\bm{S}_{b_k}|>4R\sqrt{b_k}\}}
        \right]\\
        &\leq    
        \sum_{(a_l)_{l=1}^L\in\CutTimes_L(\varepsilon)}
        E\!
        \left[
          \left(  \prod_{l=1}^L\mathbf 1_D(\bm{S}_{a_l})\right)
            \mathbf 1_{\{|\bm{S}_{b_k} - \bm{S}_{a_k}|>2R\sqrt{a_k}\}}
        \right]\\
        &\lesssim_{f,L} 
        \frac{1}{R^2} (\log N)^L.
    \end{align}
    This, combined with \eqref{eq:bound f,g,h} and Lemma~\ref{lem:prod diff}, 
    yields the first claim.
    The second claim can be verified by the same argument.
\end{proof}

Recall the sequence $(\rho_n)_{n \geq 1}$ from Lemma~\ref{lem:LCLT}.
Set, for $a, b \in \NN$, 
\begin{equation}
    \Psi(b,a)
    \coloneqq
    \sqrt{\frac{b}{a}}+\rho_b.
    \label{eq: error term function at L-1}
\end{equation}
The function $\Psi$ quantifies the error in decoupling two consecutive clusters.
The next lemma shows that this error remains negligible after summing over all admissible cluster-starting times.

\begin{lem} \label{lem:asymp indep error}
    For each $\delta \in (0,1)$, $\varepsilon \in (0,1/2)$, and $k \geq 2$,
    \begin{align}
        &(\log N )^{-1}\hspace{-1em}
        \sum_{(a_l)_{l=1}^k \in \CutTimes_k(\varepsilon)}
        a_k ^{-1}   \Psi(b_{k-1},  a_k)
        \left| 
            E\!
            \left[
                \prod_{l=1}^{k-1} 
                \func_N(S^1_{a_l}, a_l) 
                \mathbf{1}_{D_R}(\bm{S}^\dagger_{a_l})
                h_{N, a_l}^{(\varepsilon)}(\bm{S} \circ \theta_{a_l})
                \mathbf{1}_{\{|\bm{S}_{b_l}|\leq 4R \sqrt{b_{l}}\}}  
            \right]
        \right|\\
        &=  
        o_{f,k,\delta,\varepsilon}(1; N \to \infty).
    \end{align}
\end{lem}

\begin{proof}
    Let $(a_l)_{l=1}^k \in \CutTimes_k(\varepsilon)$.
    For each $k \geq 2$,
    by definition,
    we have 
    \begin{equation}
        b_{k-1}= a_{k-1} + \lfloor a_{k-1}^{1-\varepsilon} \rfloor,\qquad 
        a_k\geq a_{k-1} +  a_{k-1}^{1+\varepsilon},\qquad 
        a_{k-1} \wedge b_{k-1} \geq \lfloor N^\delta \rfloor.
    \end{equation}
    This, combined with the monotonicity of $(\rho_n)_{n \geq 1}$, yields 
    \begin{align}
        \Psi(b_{k-1},  a_k) 
        \leq    
        \sqrt{\frac{a_{k-1}+a_{k-1}^{1-\varepsilon}}{a_{k-1}+a_{k-1}^{1+\varepsilon}}} + \rho_{\lfloor N^\delta \rfloor}
        \leq    
        \sqrt{\frac{1+a_{k-1}^{-\varepsilon}}{1+a_{k-1}^\varepsilon}} + \rho_{\lfloor N^\delta \rfloor}
        =
        o_{\delta,\varepsilon}(1;N \to \infty).
    \end{align}
    Using this and \eqref{eq:bound f,g,h},
    we deduce that the left-hand side of the claim is bounded by 
    \begin{align}
       &o_{f,\delta,\varepsilon}(1;N\to \infty) 
       \cdot 
       (\log N)^{-k}\hspace{-1em}
       \sum_{(a_l)_{l=1}^k \in \CutTimes_k(\varepsilon)}
        a_k ^{-1}  
        E\!
        \left[
            \prod_{l=1}^{k-1} 
            \mathbf{1}_{D_R}(\bm{S}^\dagger_{a_l})
        \right]\\
        &\leq
        o_{f,\delta,\varepsilon}(1;N\to \infty) 
        \cdot 
        (\log N )^{-(k-1)}\hspace{-2em}
        \sum_{  1\leq a_1 < \cdots < a_{k-1} \leq \lfloor N^\kappa \rfloor} 
        E\!
        \left[
            \prod_{l=1}^{k-1} 
            \mathbf{1}_D(\bm{S}_{a_l})  
        \right].
    \end{align}
    Therefore, by Lemma~\ref{lem:col est}, we obtain the desired result.
\end{proof}

We now complete the proof of Proposition~\ref{prop:indep decomp}.

\begin{proof} [{Proof of Proposition~\ref{prop:indep decomp}}]
    The case $L=1$ is immediate, 
    since the two sums in the claim coincide exactly.  
    Hence, we assume $L\geq2$ below.
    Thanks to Lemma~\ref{lem:R rstrct1},  
    it is enough to show that
    \begin{align}
        &\sum_{(a_l)_{l=1}^L \in \CutTimes_L(\varepsilon)}
        E\!
        \left[
            \prod_{l=1}^{L} 
            \func_N(S^1_{a_l}, a_l) 
            \mathbf{1}_{D_R}(\bm{S}^\dagger_{a_l}) 
            h_{N, a_l}^{(\varepsilon)}(\bm{S} \circ \theta_{a_l})    
            \mathbf{1}_{\{|\bm{S}_{b_l}|\leq 4R \sqrt{b_{l}}\}}   
        \right]\\
        &= 
        \sum_{(a_l)_{l=1}^L \in \CutTimes_L(\varepsilon)}
        \prod_{l=1}^{L} 
        E\! 
        \left[
            \func_N(S^1_{a_l}, a_l)  
            \mathbf{1}_{D_R}(\bm{S}^\dagger_{a_l})
            h_{N, a_l}^{(\varepsilon)}(\bm{S} \circ \theta_{a_l})   
            \mathbf{1}_{\{|\bm{S}_{b_l}|\leq 4R \sqrt{b_{l}}\}} 
        \right]  +
        o_{f,L,\delta,\varepsilon,R}(1; N \to \infty).
    \end{align}
    We establish this by proving that, 
    for each $2 \leq k \leq L$,
    \begin{align}
        &\sum_{(a_l)_{l=1}^L \in \CutTimes_L(\varepsilon)}
        \prod_{l=k+1}^{L}
        E\!
        \left[
            \func_N(S^1_{a_l}, a_l) 
            \mathbf{1}_{D_R}(\bm{S}^\dagger_{a_l}) 
            h_{N, a_l}^{(\varepsilon)}(\bm{S} \circ \theta_{a_l}) 
            \mathbf{1}_{\{|\bm{S}_{b_l}|\leq 4R \sqrt{b_{l}}\}} 
        \right]
        \\
        &\hspace{4cm}\times 
        E\!
        \left[ 
            \prod_{l=1}^{k} 
            \func_N(S^1_{a_l}, a_l) 
            \mathbf{1}_{D_R}(\bm{S}^\dagger_{a_l})
            h_{N, a_l}^{(\varepsilon)}(\bm{S} \circ \theta_{a_l})
            \mathbf{1}_{\{|\bm{S}_{b_l}|\leq 4R \sqrt{b_{l}}\}} 
        \right]\\
        &= 
        \sum_{(a_l)_{l=1}^L \in \CutTimes_L(\varepsilon)}
        \prod_{l=k}^{L}
        E\! 
        \left[
            \func_N(S^1_{a_l}, a_l) 
            \mathbf{1}_{D_R}(\bm{S}^\dagger_{a_l}) 
            h_{N, a_l}^{(\varepsilon)}(\bm{S} \circ \theta_{a_l})  
            \mathbf{1}_{\{|\bm{S}_{b_l}|\leq 4R \sqrt{b_l}\}} 
        \right]
        \\
        &\hspace{4cm}\times 
        E\!
        \left[ 
            \prod_{l=1}^{k-1}
            \func_N(S^1_{a_l}, a_l) 
            \mathbf{1}_{D_R}(\bm{S}^\dagger_{a_l})
            h_{N, a_l}^{(\varepsilon)}(\bm{S} \circ \theta_{a_l}) 
            \mathbf{1}_{\{|\bm{S}_{b_l}|\leq 4R \sqrt{b_{l}}\}} 
        \right]
        \\
        &\qquad
        + 
        o_{ f,L, \delta, \varepsilon,R}(1; N \to \infty),
    \end{align}
    where products over the empty set are understood to be equal to $1$ as before. 
    Indeed, applying the above result successively for $k=L,\ldots,2$ gives the desired result.

    To this end, fix $k \in \{2,\ldots,L\}$.
    For $(a_l)_{l=1}^L \in \CutTimes_L(\varepsilon)$,
    it follows from the definition of $\CutTimes_L(\varepsilon)$ that,
    for all sufficiently large $N$,
    $a_k>2b_{k-1}$,
    as required for applying Proposition~\ref{prop:asymp indep}.
    Hence, 
    by the Markov property at time $b_{k-1}$ and Proposition~\ref{prop:asymp indep},
    we deduce that
    \begin{align}
        &E\!
        \left[ 
            \prod_{l=1}^{k} 
            \func_N(S^1_{a_l}, a_l) 
            \mathbf{1}_{D_R}(\bm{S}^\dagger_{a_l})
            h_{N, a_l}^{(\varepsilon)}(\bm{S} \circ \theta_{a_l})
            \mathbf{1}_{\{|\bm{S}_{b_l}|\leq 4R \sqrt{b_{l}}\}} 
        \right]\\
        &= 
        E
        \Biggl[
                \prod_{l=1}^{k-1} 
                \func_N(S^1_{a_l}, a_l) 
                \mathbf{1}_{D_R}(\bm{S}^\dagger_{a_l})
                h_{N, a_l}^{(\varepsilon)}(\bm{S} \circ \theta_{a_l})
            \mathbf{1}_{\{|\bm{S}_{b_l}|\leq 4R \sqrt{b_{l}}\}} 
        \Biggr.
        \\
        &\hspace{4cm}\times 
        \Biggl.
            E\!
            \left[
                \func_N(S^1_{a_k}, a_k)
                \mathbf{1}_{D_R}(\bm{S}^\dagger_{a_k})
                h_{N, a_k}^{(\varepsilon)}(\bm{S} \circ \theta_{a_k})
                \mathbf{1}_{\{|\bm{S}_{b_k}|\leq 4R \sqrt{b_k}\}} 
                \, \middle \vert\, 
                \bm{S}_{b_{k-1}}
            \right]
        \Biggr]\\
        &= 
        E\!
        \left[
            \prod_{l=1}^{k-1} 
            \func_N(S^1_{a_l}, a_l) 
            \mathbf{1}_{D_R}(\bm{S}^\dagger_{a_l})
            h_{N, a_l}^{(\varepsilon)}(\bm{S} \circ \theta_{a_l}) 
            \mathbf{1}_{\{|\bm{S}_{b_l}|\leq 4R \sqrt{b_{l}}\}}
        \right]\\
        &\hspace{4cm}\times
        E\!
        \left[
            \func_N(S^1_{a_k}, a_k)
            \mathbf{1}_{D_R}(\bm{S}^\dagger_{a_k})
            h_{N, a_k}^{(\varepsilon)}(\bm{S} \circ \theta_{a_k})
            \mathbf{1}_{\{|\bm{S}_{b_k}|\leq 4R \sqrt{b_k}\}} 
        \right]\\
        &\quad 
        + 
        \frac{O_{f,R}(1)}{\log N} 
        a_k^{-1}  \Psi(b_{k-1},a_k)
        E\!
        \left[
            \prod_{l=1}^{k-1} 
            \func_N(S^1_{a_l}, a_l) 
            \mathbf{1}_{D_R}(\bm{S}^\dagger_{a_l})
            h_{N, a_l}^{(\varepsilon)}(\bm{S} \circ \theta_{a_l})\mathbf{1}_{\{|\bm{S}_{b_l}|\leq 4R \sqrt{b_{l}}\}}
        \right].
    \end{align}
    By Lemma \ref{lem:col est},
    it follows that
    \begin{align}
        &\sum_{1 \leq a_{k+1} < \cdots < a_L \leq \lfloor N^\kappa \rfloor}
        \prod_{l=k+1}^{L}
        \left|
            E\!
            \left[
                \func_{N}(S^1_{a_l}, a_l) 
                \mathbf{1}_{D_R}(\bm{S}^\dagger_{a_l}) 
                h_{N, a_l}^{(\varepsilon)}(\bm{S} \circ \theta_{a_l})
            \right]
        \right|\\
        &\lesssim_{f,L}
        (\log N)^{-(L-k)}\hspace{-2em}
        \sum_{1 \leq a_{k+1} < \cdots < a_L \leq \lfloor N^\kappa \rfloor}
        \prod_{l=k+1}^{L}
        P(S^1_{a_l} = S^2_{a_l})\\
        &\lesssim_{f,L}
        1.
    \end{align}
    This yields
    \begin{align}
        &\sum_{(a_l)_{l=1}^L \in \CutTimes_L(\varepsilon)} 
        \prod_{i=k+1}^{L}
        \left|
            E\!
            \left[
                \func_N(S^1_{a_i}, a_i) 
                \mathbf{1}_{D_R}(\bm{S}^\dagger_{a_i}) 
                h_{N, a_i}^{(\varepsilon)}(\bm{S} \circ \theta_{a_i})
                \mathbf{1}_{\{|\bm{S}_{b_i}|\leq 4R \sqrt{b_i}\}}
            \right]
        \right|
        \\
        &\hspace{3cm}\times 
        a_k^{-1}  \Psi(b_{k-1},a_k)
        \left|
            E\!
            \left[
                \prod_{l=1}^{k-1} 
                \func_N(S^1_{a_l}, a_l) 
                \mathbf{1}_{D_R}(\bm{S}^\dagger_{a_l})
                h_{N, a_l}^{(\varepsilon)}(\bm{S} \circ \theta_{a_l})
                \mathbf{1}_{\{|\bm{S}_{b_l}|\leq 4R \sqrt{b_{l}}\}}
            \right]
        \right|
        \\
        &\lesssim_{f,L,\delta}
        \sum_{(a_l)_{l=1}^k \in \CutTimes_k(\varepsilon)}\hspace{-1.4em}
        a_k^{-1} \Psi(b_{k-1},a_k)
        \left| 
            E\!
            \left[
                \prod_{l=1}^{k-1} 
                \func_N(S^1_{a_l}, a_l) 
                \mathbf{1}_{D_R}(\bm{S}^\dagger_{a_l})
                h_{N, a_l}^{(\varepsilon)}(\bm{S} \circ \theta_{a_l})\mathbf{1}_{\{|\bm{S}_{b_l}|\leq 4R \sqrt{b_{l}}\}}
            \right]
        \right|.
    \end{align}
    Applying Lemma~\ref{lem:asymp indep error},
    we obtain the desired result.
\end{proof}

Below, we state a consequence of Corollary~\ref{cor:neg col} and Proposition~\ref{prop:indep decomp}.

\begin{cor}\label{cor:indep decomp}
    For each $\delta \in (0,1)$,
    \begin{align}
        &\sum_{L=1}^{\infty}
        \sum_{(a_l)_{l=1}^L \in \CutTimes_L(-\varepsilon)}\!
        E\!
        \left[
            \prod_{l=1}^{L} 
            \func_N(S^1_{a_l}, a_l) 
            \mathbf{1}_D(\bm{S}_{a_l})
            h^{(\varepsilon)}_{N,a_l}(\bm{S} \circ \theta_{a_l})
        \right]\\
        &= 
        \sum_{L=1}^M
        \sum_{(a_l)_{l=1}^L \in \CutTimes_L(\varepsilon)}
        \prod_{l=1}^{L} 
        E\!
        \left[
            \func_N(S^1_{a_l}, a_l) 
            \mathbf{1}_{D_R}(\bm{S}^\dagger_{a_l})
            h_{N,a_l}^{(\varepsilon)}(\bm{S} \circ \theta_{a_l})
        \right]\\
        &\qquad
        +
        o_{f,\delta}(1; N \to \infty\, \triangleright\, R \to \infty\, \triangleright\, \varepsilon \to 0\, \triangleright\, M \to \infty).
        \label{eq:indep decomp}
    \end{align}
\end{cor}

\subsection{Approximation within a single cluster}

In the cluster decomposition of \eqref{eq:indep decomp},
\begin{align}
    &\func_N(S^1_{a_l}, a_l) 
    \mathbf{1}_{D_R}(\bm{S}^\dagger_{a_l})
    h_{N,a_l}^{(\varepsilon)}(\bm{S} \circ \theta_{a_l})\\
    &=
    \func_N(S^1_{a_l}, a_l) 
    \mathbf{1}_{D_R}(\bm{S}^\dagger_{a_l})
    \exp
    \biggl(
        - 
        \sum_{k=a_l+1}^{a_l + \lfloor a_l^{1-\varepsilon} \rfloor} 
        f_N(S^1_k, k) 
        \mathbf{1}_D(\bm{S}_k)
    \biggr)
\end{align}
represents the contribution from the $l$-th cluster.
In this subsection, we aim to freeze the space and time arguments appearing in $f_N$ at the start of each cluster, that is, we replace
 $f_N(S^1_k,k)$ by $f_N(S^1_{a_l},a_l)$.
The following proposition makes this precise.

\begin{prop}\label{prop:time approx}
    For each $L \geq 1$, $\delta \in (0,1)$, $\varepsilon \in (0,1/2)$, and $R > 1$,
    \begin{align}
        &\sum_{(a_l)_{l=1}^L \in \CutTimes_L(\varepsilon)}
        \prod_{l=1}^{L} 
        E\! 
        \left[
            \func_N(S^1_{a_l}, a_l) 
            \mathbf{1}_{D_R}(\bm{S}^\dagger_{a_l})
            \exp 
            \biggl(
                - \sum_{k=a_l+1}^{a_l + \lfloor a_l^{1-\varepsilon} \rfloor} 
                f_N(S^1_k, k)\, \mathbf{1}_D(\bm{S}_k)
            \biggr)
        \right]
        \\
        &=
        \sum_{(a_l)_{l=1}^L \in \CutTimes_L(\varepsilon)}
        \prod_{l=1}^{L} 
        E\!
        \left[
            \func_N(S^1_{a_l}, a_l) 
            \mathbf{1}_{D_R}(\bm{S}^\dagger_{a_l})
            \exp 
            \biggl(
                - f_N(S^1_{ a_l}, a_l)
                \sum_{k=a_l+1}^{a_l + \lfloor a_l^{1-\varepsilon} \rfloor} 
                \mathbf{1}_D(\bm{S}_k)
            \biggr)
        \right]
        \\
        &\qquad
        +
        o_{f, L, \delta, \varepsilon, R}(1; N \to \infty).
    \end{align}
\end{prop}

We recall the scaling map $\phi_N$ from \eqref{eq:phi}. 
The proof of Proposition~\ref{prop:time approx} uses the following continuity estimate for $\phi_N$.

\begin{lem} \label{lem:phi_N conti}
    For any $x,y \in \ZN^2$ and $s, t \in \NN$,
    \begin{equation}
        |\phi_N(x, s) - \phi_N(y, t)|
        \leq 
        \frac{|x-y|}{\sqrt{s}}
        +
        \frac{|y|\,|s-t|}{2(s\wedge t)^{3/2}}
        + 
        \frac{|s-t|}{(s \wedge t) \log N}.
    \end{equation}
\end{lem}

\begin{proof}
    By definition,
    we have 
    \begin{equation}
        |\phi_N(x, s) - \phi_N(y, t)|
        \leq 
        \left|
            \frac{x}{\sqrt{s}} - \frac{y}{\sqrt{t}}
        \right|
        +
        \left|
            \frac{\log s}{\log N} - \frac{\log t}{\log N}
        \right|.
    \end{equation}
    By the triangle inequality,
    \begin{equation}
        \left|
            \frac{x}{\sqrt{s}}-\frac{y}{\sqrt{t}}
        \right|
        \leq
        \frac{|x-y|}{\sqrt{s}}
        +
        |y|
        \left|
            \frac{1}{\sqrt{s}}-\frac{1}{\sqrt{t}}
        \right|.
    \end{equation}
    The mean value theorem gives
    \begin{equation}
        \left|
            \frac{1}{\sqrt{s}}-\frac{1}{\sqrt{t}}
        \right|
        \leq
        \frac{|s-t|}{2(s\wedge t)^{3/2}},
        \qquad 
        \left| \frac{\log s}{\log N} - \frac{\log t}{\log N} \right| 
        \leq 
        \frac{|s-t|}{(s \wedge t) \log N}.
    \end{equation}
    This completes the proof.
\end{proof}

\begin{proof}[Proof of Proposition~\ref{prop:time approx}]
    Fix $L \geq 1$, $\delta \in (0,1)$, $\varepsilon \in (0,1/2)$, and $R > 1$.
    For $N \geq 2$ and $a \in \NN$,
    we write 
    \begin{gather}
        \zeta_N(a)
        \coloneqq 
        \exp 
        \biggl(
            - \sum_{k=a+1}^{a + \lfloor a^{1-\varepsilon} \rfloor} 
            f_N(S^1_k, k)\, \mathbf{1}_D(\bm{S}_k)
        \biggr),
        \\
        \zeta'_N(a)
        \coloneqq
        \exp 
        \biggl(
            - f_N(S^1_a, a)
            \sum_{k=a+1}^{a + \lfloor a^{1-\varepsilon} \rfloor} 
            \mathbf{1}_D(\bm{S}_k)
        \biggr).
    \end{gather}
    By \eqref{eq:e^-x conti},
    we have  
    \begin{equation} \label{eq:time approx1}
        |\zeta_N(a) - \zeta'_N(a)|
        \leq 
        \sum_{k=a+1}^{a + \lfloor a^{1-\varepsilon} \rfloor} 
        \bigl| f_N(S^1_k, k) - f_N(S^1_a, a) \bigr|\, 
        \mathbf{1}_D(\bm{S}_k).
    \end{equation}
    Let $N$ be large and let $a \in \NN$ satisfy $N^\delta \leq a \leq N^\kappa$.
    We define an event by
    \begin{equation}
        E(a) \coloneqq  
        \bigl\{ 
            |S_k^1 -S_a^1| < a^{\frac{1}{2}-\frac{\varepsilon}{4}},\, 
            \forall k\in [a+1,a + \lfloor a^{1-\varepsilon} \rfloor] \cap \NN 
        \bigr\}. 
    \end{equation}
    By Kolmogorov's maximal inequality \cite[Theorem~2.5.5, p.~84]{Durret_19_Probability} and the Markov property, 
    we have 
    \begin{equation} \label{eq:time approx2}
        \sup_{\bm{x} \in \ZN^4} P_{\bm{x}}(E(a)^c) 
        = O(1) a^{-\varepsilon/2}
        = o_{\delta, \varepsilon}(1; N \to \infty).
    \end{equation}
    On the event $E(a)\cap\{\bm{S}^\dagger_a\in D_R\}$, Lemma~\ref{lem:phi_N conti} yields that, as long as $a+1 \leq k \leq a + \lfloor a^{1-\varepsilon} \rfloor$,
    \begin{align}
        \bigl| \phi_N(S^1_k, k) - \phi_N(S^1_a, a) \bigr|
        &\leq    
        \frac{|S^1_k - S^1_a|}{\sqrt{k}}
        +
        \frac{|S^1_a|\,|k-a|}{2(k \wedge a)^{3/2}}
        + 
        \frac{|k-a|}{(k \wedge a) \log N}\\
        &\lesssim_R
        a^{-\varepsilon/4} + 2a^{-\varepsilon}\\
        &= 
        o_{\delta, \varepsilon, R}(1; N \to \infty),
    \end{align}
    where the last asymptotic follows from $N^\delta \leq a \leq N^\kappa$.
    Recalling that $f_N = (\log N)^{-1} f \circ \phi_N$
    and $f$ is uniformly continuous,
    we obtain
    \begin{equation}
        \bigl| f_N(S^1_k, k) - f_N(S^1_a, a) \bigr|
        \leq    
        \frac{o_{f,\delta, \varepsilon, R}(1; N \to \infty)}{\log N}.
    \end{equation}
    From this and \eqref{eq:time approx1},
    we deduce that, on the event $E(a)\cap\{\bm{S}^\dagger_a\in D_R\}$,
    \begin{equation}
        |\zeta_N(a) - \zeta'_N(a)|
        \leq
        \frac{o_{f,\delta,\varepsilon,R}(1; N \to \infty)}{\log N}
        \sum_{k=a+1}^{a + \lfloor a^{1-\varepsilon} \rfloor} 
        \mathbf{1}_D(\bm{S}_k).
    \end{equation}
    Combining this with \eqref{eq:bound f,g,h}, we obtain
    \begin{align}
        &\sum_{a = \lceil N^\delta \rceil}^{\lfloor N^\kappa \rfloor}
        \left|
            E\!\left[ \func_N(S^1_a, a)\, \mathbf{1}_{D_R}(\bm{S}^{\dagger}_a)\, \zeta_N(a) \right]
            -
            E\!\left[ \func_N(S^1_a, a)\, \mathbf{1}_{D_R}(\bm{S}^{\dagger}_a)\, \zeta'_N(a) \right]
        \right|
        \\
        &\lesssim_f 
        \sum_{a = \lceil N^\delta \rceil}^{\lfloor N^\kappa \rfloor}
        \frac{o_{f,\delta,\varepsilon,R}(1; N \to \infty)}{(\log N)^2} 
        E\! 
        \left[
            \mathbf{1}_{D_R}(\bm{S}^{\dagger}_a) \,            
            \sum_{k=a+1}^{a + \lfloor a^{1-\varepsilon} \rfloor} 
            \mathbf{1}_D(\bm{S}_k)\,
            \mathbf{1}_{E(a)}
        \right]\\
        &\qquad
        +  
        \sum_{a = \lceil N^\delta \rceil}^{\lfloor N^\kappa \rfloor}
        \frac{1}{\log N} 
        E\!\left[ \mathbf{1}_{D_R}(\bm{S}^{\dagger}_a) \cdot\,\mathbf{1}_{E(a)^c} \right].
        \label{eq:time approx3}
    \end{align}
    By the Markov property at time $a$ and Lemma~\ref{lem:col est},
    the first term on the right-hand side is bounded by 
    \begin{align}
        &\frac{o_{f,\delta,\varepsilon,R}(1; N \to \infty)}{(\log N)^2} 
        \sum_{a = \lceil N^\delta \rceil}^{\lfloor N^\kappa \rfloor}
        P\bigl(S^1_a = S^2_a\bigr)
        \sup_{x \in \ZN^2}
        E_{(x,x)}\!
        \left[   
            \sum_{k=1}^{\lfloor a^{1-\varepsilon} \rfloor} 
            \mathbf{1}_D(\bm{S}_k)
        \right]\\
        &\lesssim
        \frac{o_{f,\delta,\varepsilon,R}(1; N \to \infty)}{\log N} 
        \sum_{a = \lceil N^\delta \rceil}^{\lfloor N^\kappa \rfloor}
        P\bigl(S^1_a = S^2_a\bigr)\\
        &=
        o_{f,\delta,\varepsilon,R}(1; N \to \infty).
         \label{eq:time approx4}
    \end{align}
    Similarly, by the Markov property at time $a$, \eqref{eq:time approx2}, and Lemma~\ref{lem:col est},
    the second term of \eqref{eq:time approx3} is bounded by 
    \begin{equation}
        \sum_{a = \lceil N^\delta \rceil}^{\lfloor N^\kappa \rfloor}
        \frac{1}{\log N} 
        P\bigl(S^1_a = S^2_a\bigr)
        \sup_{x \in \ZN^2}
        P_{(x,x)}(E(a)^c)
        = o_{f,\delta,\varepsilon}(1;N\to \infty).
    \end{equation}
    Hence, by Lemmas~\ref{lem:col est} and~\ref{lem:prod diff}, 
    we conclude 
    \begin{align}
        &\left|\sum_{(a_l)_{l=1}^L \in \CutTimes_L(\varepsilon)}
            \left(
                \prod_{l = 1}^L
                E\!
                \left[ 
                    \func_N(S^1_{a_l}, a_l)\, \mathbf{1}_{D_R}(\bm{S}^{\dagger}_{a_l})\, \zeta_N(a_l) 
                \right]
                -  
                \prod_{l = 1}^L
                E\! 
                \left[
                    \func_N(S^1_{a_l}, a_l)\,
                    \mathbf{1}_{D_R}(\bm{S}^{\dagger}_{a_l})\,
                    \zeta'_N(a_l)
                \right]
            \right)
        \right|\\
        &\leq    
        \sum_{k=1}^L 
        o_{f,\delta,\varepsilon,R}(1;N\to \infty)
        \left(\frac{\|f\|_\infty}{\log N}\right)^{L-1}   
        \left( 
            \sum_{a= \lceil N^{\delta} \rceil}^{\lfloor N^\kappa \rfloor} 
            E\bigl[ \mathbf{1}_{D_R}(\bm{S}^{\dagger}_{a}) \bigr]
        \right)^{L-1}\\
        &=
        o_{f,L,\delta,\varepsilon,R}(1; N \to \infty),
    \end{align}
    which completes the proof.
\end{proof}

Below, we state a consequence of Corollary~\ref{cor:indep decomp} and Proposition~\ref{prop:time approx}.

\begin{cor} \label{cor:clust approx}
    For each $\delta \in (0,1)$,
    \begin{align}
        &\sum_{L=1}^{\infty}
        \sum_{(a_l)_{l=1}^L \in \CutTimes_L(-\varepsilon)}\!
        E\!
        \left[
            \prod_{l=1}^{L} 
            \func_N(S^1_{a_l}, a_l) 
            \mathbf{1}_D(\bm{S}_{a_l})
            h^{(\varepsilon)}_{N,a_l}(\bm{S} \circ \theta_{a_l})
        \right]\\
        &= 
        \sum_{L=1}^M
        \sum_{(a_l)_{l=1}^L \in \CutTimes_L(\varepsilon)}
        \prod_{l=1}^{L} 
        E\!
        \left[
            \func_N(S^1_{a_l}, a_l) 
            \mathbf{1}_{D_R}(\bm{S}^{\dagger}_{a_l})
            \exp 
            \biggl(
                -f_N(S^1_{a_l}, a_l)
                \sum_{k=a_l+1}^{a_l + \lfloor a_l^{1-\varepsilon} \rfloor} 
                \mathbf{1}_D(\bm{S}_k)
            \biggr)
        \right]\\[0.4em]
        &\qquad
        +
        o_{f,\delta}(1; N \to \infty\, \triangleright\, R \to \infty\, \triangleright\, \varepsilon \to 0\, \triangleright\, M \to \infty).
    \end{align}
\end{cor}

\subsection{Proof of Theorem~\ref{thm:lap conv}}

In this subsection,
we combine the preceding approximations to identify the limiting Laplace transform.
We first simplify each single-cluster contribution and then convert the sum over
cluster-starting times into an ordered time integral.

We begin by removing the annulus restrictions, which were needed for preceding approximation steps.

\begin{lem}\label{lem:remove annu}
    For each $N \geq 2$, $L \geq 1$, $\delta \in (0,1)$, and $\varepsilon \in (0,1/2)$,
    \begin{align}
        &\sum_{(a_l)_{l=1}^L \in \CutTimes_L(\varepsilon)}
        \prod_{l=1}^{L} 
        E\!
        \left[
            \func_N(S^1_{a_l}, a_l)\, 
            \mathbf{1}_{D_R}(\bm{S}^\dagger_{a_l})
            \exp 
            \biggl(
                - f_N(S^1_{a_l}, a_l)
                \sum_{k=a_l+1}^{a_l + \lfloor a_l^{1-\varepsilon} \rfloor} 
                \mathbf{1}_D(\bm{S}_k)
            \biggr)
        \right]
        \\
        &=    
        \sum_{(a_l)_{l=1}^L \in \CutTimes_L(\varepsilon)}
        \prod_{l=1}^{L} 
        E\!
        \left[
            \func_N(S^1_{a_l}, a_l)\, 
            \mathbf{1}_D(\bm{S}_{a_l})
            \exp 
            \biggl(
                - f_N(S^1_{a_l}, a_l)
                \sum_{k=a_l+1}^{a_l + \lfloor a_l^{1-\varepsilon} \rfloor} 
                \mathbf{1}_D(\bm{S}_k)
            \biggr)
        \right]
        \\ 
        &\qquad 
        + o_{f,L}(1; N \to \infty\, \triangleright\, R \to \infty).
    \end{align}
\end{lem}

\begin{proof}
    For $a \in [N^\delta,N^\kappa]\cap\ZN$, 
    let $A_{N,R}(a)$ and $A_N(a)$ denote the expectations 
    appearing in the first and second sums, respectively.  
    Since the exponential factor is bounded by $1$,
    using \eqref{eq:bound f,g,h},
    we obtain 
    \begin{align}
        |A_{N,R}(a)-A_N(a)|
        &\leq
        \frac{\|f\|_\infty}{\log N}
        P\!\left(\bm{S}_a\in D,\, \bm{S}_a^\dagger\notin D_R\right),
        \label{eq:final annulus one factor}
        \\[0.2em]
        |A_{N,R}(a)|\vee|A_N(a)|
        &\leq
        \frac{\|f\|_\infty}{\log N}P(\bm{S}_a\in D).
    \end{align}
    Combining this with Lemma~\ref{lem:prod diff},
    we deduce that
    \begin{align}
        &\left|
            \sum_{(a_l)_{l=1}^L \in \CutTimes_L(\varepsilon)}
            \prod_{l=1}^L A_{N,R}(a_l) 
            - 
            \sum_{(a_l)_{l=1}^L \in \CutTimes_L(\varepsilon)}
            \prod_{l=1}^L A_N(a_l)
        \right|\\
        &\lesssim_{f,L}
        \frac{1}{(\log N)^L}
       \left( \sum_{a = 1}^{\lfloor N^\kappa \rfloor}
        P\!\left( \bm{S}_a\in D,\, \bm{S}_a^\dagger\notin D_R \right) \right) 
        \left( \sum_{a = 1}^{\lfloor N^\kappa \rfloor} P(\bm{S}_a\in D) \right)^{L-1}\\
        &\lesssim_{f,L}
        \frac{1}{\log N}+\frac{1}{R^2},
    \end{align}
    where we used Lemma~\ref{lem:annu neg} with $L=1$ and Lemma~\ref{lem:col est}
    to obtain the last inequality.
    This completes the proof.
\end{proof}

The exponential factor in each cluster records the collisions following its starting time.
We next replace its Laplace transform by the limiting function obtained in Corollary~\ref{cor:col num conv}. 
Recalling the constant $\tilde{C}$ from \eqref{eq:Charact},  
set
\begin{equation}
    \psi(\lambda, t) 
    \coloneqq 
    \frac{\Charact}{\lambda\, t + \Charact},
    \qquad 
    \lambda, t \geq 0.
\end{equation}

\begin{lem}\label{lem:clust Lap approx}
    For each $L \geq 1$, $\delta \in (0,1)$, and $\varepsilon \in (0,1/2)$,
    \begin{align}
        &\sum_{(a_l)_{l=1}^L \in \CutTimes_L(\varepsilon)}
        \prod_{l=1}^{L} 
        E\!
        \left[
            \func_N(S^1_{a_l}, a_l)\, 
            \mathbf{1}_D(\bm{S}_{a_l})
            \exp 
            \biggl(
                - f_N(S^1_{a_l}, a_l)
                \sum_{k=a_l+1}^{a_l + \lfloor a_l^{1-\varepsilon} \rfloor} 
                \mathbf{1}_D(\bm{S}_k)
            \biggr)
        \right]
        \\
        &=    
        \sum_{(a_l)_{l=1}^L \in \CutTimes_L(\varepsilon)}
        \prod_{l=1}^{L} 
        E\!
        \left[
            \func_N(S^1_{a_l}, a_l)\,
            \mathbf{1}_D(\bm{S}_{a_l})\,
            \psi
            \!\left(
                f\bigl( \phi_N(S^1_{a_l}, a_l) \bigr), 
                \frac{(1-\varepsilon)\log a_l}{\log N}
            \right)
        \right]
        \\ 
        &\qquad 
        + o_{f,L, \delta, \varepsilon}(1; N \to \infty).
    \end{align}
\end{lem}

\begin{proof}
    Set
    \begin{equation}
        \psi_N(\lambda, t) 
        \coloneqq 
        E\!
        \left[
            \exp\!\left( -\frac{\lambda}{\log N} \sum_{n=1}^{\lfloor N^t \rfloor} \mathbf{1}_D(\bm{S}_n) \right)
        \right],
        \qquad 
        \lambda \geq 0,\, t  > 0.
    \end{equation}
    By the Markov property at time $a=a_l$ and translation invariance, 
    the expectation on the left-hand side in the claim equals
    \begin{equation}
        E\!\left[
            \func_N(S^1_a,a)\mathbf{1}_D(\bm{S}_a)
            \psi_N\!\left( 
                f\bigl( \phi_N(S^1_a, a) \bigr),
                \frac{\log\lfloor a^{1-\varepsilon}\rfloor}{\log N} 
            \right)
        \right].
        \label{eq:clust Lap approx1}
    \end{equation}
    For all sufficiently large $N$, 
    uniformly in $a \in [N^\delta,N^\kappa] \cap \ZN$, 
    we have
    \begin{equation}
        \frac{\delta}{2} 
        \leq    
        \frac{\log\lfloor a^{1-\varepsilon}\rfloor}{\log N}
        \leq    
        \kappa,
        \qquad
        \left|
            \frac{\log\lfloor a^{1-\varepsilon}\rfloor}{\log N}
            -
            \frac{(1-\varepsilon)\log a}{\log N}
        \right|=o_\delta(1;N\to\infty).
    \end{equation}
    Therefore, 
    Corollary~\ref{cor:col num conv} and 
    the uniform continuity of $\psi$ on $[0,\|f\|_\infty]\times[\delta/2,\kappa]$ yield
    \begin{align}
        \eta_N
        \coloneqq
        \sup_{\substack{N^\delta\leq a\leq N^\kappa\\
                        0\leq\lambda\leq\|f\|_\infty}}
        \left|
            \psi_N\!\left( \lambda, \frac{\log\lfloor a^{1-\varepsilon}\rfloor}{\log N} \right)
            -\psi\biggl(\lambda, \frac{(1-\varepsilon)\log a}{\log N}\biggr)
        \right|
        =o_{f,\delta,\varepsilon}(1;N\to\infty).
    \end{align}
    Let $A_N(a)$ denote \eqref{eq:clust Lap approx1}, 
    and let 
    \begin{equation}
        B_N(a)
        \coloneqq 
        E\!\left[
            \func_N(S^1_a,a)\mathbf{1}_D(\bm{S}_a)
            \psi\!\left( f\bigl( \phi_N(S^1_a, a) \bigr), \frac{(1-\varepsilon)\log a}{\log N} \right)
        \right].
    \end{equation}
    Then, by \eqref{eq:bound f,g,h} and the definition of $\eta_N$,
    \begin{equation}
        |A_N(a)-B_N(a)|
        \leq
        \frac{\|f\|_\infty\eta_N}{\log N}P(\bm{S}_a\in D),
        \qquad
        |A_N(a)|\vee|B_N(a)|
        \leq
        \frac{\|f\|_\infty}{\log N}P(\bm{S}_a\in D).
    \end{equation}
    Using this and Lemma~\ref{lem:prod diff},
    we deduce that 
    \begin{align}
        &\left|
            \sum_{(a_l)_{l=1}^L \in \CutTimes_L(\varepsilon)}
            \prod_{l=1}^L A_N(a_l) 
            - 
            \sum_{(a_l)_{l=1}^L \in \CutTimes_L(\varepsilon)}
            \prod_{l=1}^L B_N(a_l)
        \right|\\
        &\lesssim_{f,L}
        \frac{\eta_N}{(\log N)^L}
        \left( \sum_{a = 1}^{\lfloor N^\kappa \rfloor} P(\bm{S}_a\in D) \right)^L\\
        &\lesssim_{f,L}
        \eta_N,
    \end{align}
    where we used Lemma~\ref{lem:col est} to obtain the last inequality.
    This completes the proof.
\end{proof}

Using the approximation $e^{-u} - 1 \approx -u$ for small $u \geq 0$,
we expect
\begin{equation}
    \func_N(S^1_{a_l}, a_l)
    = 
    e^{-f_N(S^1_{a_l}, a_l)} - 1
    \approx
    -f_N(S^1_{a_l}, a_l).
\end{equation}
In the following lemma,
we justify this approximation.

\begin{lem}\label{lem:linearize}
    For each $L \geq 1$, $\delta \in (0,1)$, and $\varepsilon \in (0,1/2)$,
    \begin{align}
        &\sum_{(a_l)_{l=1}^L \in \CutTimes_L(\varepsilon)}
        \prod_{l=1}^{L} 
        E\!
        \left[
            \func_N(S^1_{a_l}, a_l)\,
            \mathbf{1}_D(\bm{S}_{a_l})\,
            \psi
            \!\left(
                f\bigl( \phi_N(S^1_{a_l}, a_l) \bigr), 
                \frac{(1-\varepsilon)\log a_l}{\log N}
            \right)
        \right]\\ 
        &=
        \sum_{(a_l)_{l=1}^L \in \CutTimes_L(\varepsilon)}
        \prod_{l=1}^{L} 
        E\!
        \left[
            -f_N(S^1_{a_l}, a_l)\,
            \mathbf{1}_D(\bm{S}_{a_l})\,
            \psi
            \!\left(
                f\bigl( \phi_N(S^1_{a_l}, a_l) \bigr), 
                \frac{(1-\varepsilon)\log a_l}{\log N}
            \right)
        \right]\\
        &\qquad
        +
        o_{f,L}(1; N \to \infty).
    \end{align}
\end{lem}

\begin{proof}
    We use the following elementary bound: 
    \begin{equation}
        |(e^{-u} - 1) - (-u)| \leq \frac{u^2}{2}, \quad \forall u \geq 0.
    \end{equation}
    This, combined with \eqref{eq:bound f,g,h}, yields 
    \begin{equation}
        \sup_{x \in \ZN^2,\, a \in \NN}
        \left|
            \func_N(x,a) - (-f_N(x,a)) 
        \right|
        \leq
        \frac{\|f_N\|_\infty^2}{2}
        \leq
        \frac{\|f\|_\infty^2}{2(\log N)^2}.
        \label{eq:linearize1}
    \end{equation}
    For $a \in [N^\delta,N^\kappa]\cap\ZN$, 
    let $A_N(a)$ and $B_N(a)$ denote the expectations 
    appearing in the first and second sums in the claim, respectively.  
    Then, we deduce from \eqref{eq:bound f,g,h}, \eqref{eq:linearize1} and $\|\psi\|_\infty \leq 1$ that 
    \begin{equation}
        |A_N(a)-B_N(a)|
        \leq
        \frac{\|f\|_\infty^2}{(\log N)^2}P(\bm{S}_a\in D),
        \qquad
        |A_N(a)|\vee|B_N(a)|
        \leq
        \frac{\|f\|_\infty}{\log N}P(\bm{S}_a\in D).
    \end{equation}
    Therefore,
    analogously to the proof of Lemma~\ref{lem:clust Lap approx},
    we obtain
    \begin{align}
        &\left|
            \sum_{(a_l)_{l=1}^L \in \CutTimes_L(\varepsilon)}
            \prod_{l=1}^L A_N(a_l) 
            - 
            \sum_{(a_l)_{l=1}^L \in \CutTimes_L(\varepsilon)}
            \prod_{l=1}^L B_N(a_l)
        \right|\\
        &\lesssim_{f,L}
        \frac{1}{(\log N)^{L+1}}
        \left( \sum_{a = 1}^{\lfloor N^\kappa \rfloor} P(\bm{S}_a\in D) \right)^L\\
        &\lesssim_{f,L}
        \frac{1}{\log N},
    \end{align}
    which implies the desired result.
\end{proof}

It remains to identify the limiting contribution of a collision
at a cluster-starting time.
The next lemma uses Corollary~\ref{cor:sum to int}
to replace the conditional distribution of the rescaled collision position
at each cluster-starting time by its limiting law $\nu_{\Gamma/2}$.
Set, for each $\varepsilon \in [0,1]$, $t \geq 0$, and $x \in \RN^2$,
\begin{equation}
    F_\varepsilon(x, t)  
    \coloneqq 
    -f(x, t)\, 
    \psi\bigl( f(x, t), (1-\varepsilon)t  \bigr)
    = 
    -\frac{\Charact f(x,t)}{(1-\varepsilon)t f(x,t) + \Charact}.
\end{equation}

\begin{lem} \label{lem:circle int}
    For each $L \geq 1$, $\delta \in (0,1)$, and $\varepsilon \in (0,1/2)$,
    \begin{align}
        &\sum_{(a_l)_{l=1}^L \in \CutTimes_L(\varepsilon)}
        \prod_{l=1}^{L} 
        E\!
        \left[
            -f_N(S^1_{a_l}, a_l)\,
            \mathbf{1}_D(\bm{S}_{a_l})\,
            \psi
            \!\left(
                f\bigl( \phi_N(S^1_{a_l}, a_l) \bigr), 
                \frac{(1-\varepsilon)\log a_l}{\log N}
            \right)
        \right]
        \\
        &=    
        \left(\frac{c_*}{2}\right)^L
        \sum_{(a_l)_{l=1}^L \in \CutTimes_L(\varepsilon)}
        \prod_{l=1}^{L} 
        \left\{
            \frac{1}{a_l \log N}
            \int_{\RN^2}
            F_\varepsilon\left( x, \frac{\log a_l}{\log N} \right) 
            \nu_{\Gamma/2}(dx)
        \right\} 
        + o_{f,L,\delta,\varepsilon}(1;N\to \infty).
    \end{align}
\end{lem}

\begin{proof}
    The definition of $F_\varepsilon$ yields
    \begin{equation}
        -f_N(S^1_a, a)\,
        \psi
        \!\left(
            f\bigl( \phi_N(S^1_{a}, a) \bigr), 
            \frac{(1-\varepsilon)\log a}{\log N}
        \right)
        = 
        \frac{1}{\log N}
        F_\varepsilon\!
        \left(
            \frac{S^1_a}{\sqrt{a}}, \frac{\log a}{\log N}
        \right).
    \end{equation}
    Hence, the expectation on the left-hand side of the claim equals
    \begin{equation}
        \frac{1}{\log N}
        E\!\left[
            F_\varepsilon\!
            \left(
                \frac{S^1_{a_l}}{\sqrt{a_l}}, \frac{\log a_l}{\log N}
            \right)
            \mathbf{1}_D(\bm{S}_{a_l})
        \right].
    \end{equation}
    Observe that
    \begin{equation}
        |\lambda \psi(\lambda,t) - \lambda'\psi(\lambda',t)| \leq |\lambda-\lambda'|,
        \qquad 
        \forall \lambda,\lambda',t \geq 0.
    \end{equation}
    This yields 
    \begin{equation}
        \left|
            F_\varepsilon(x, t) - F_\varepsilon(x',t)
        \right|
        \leq    
        \bigl| f(x, t) - f(x', t) \bigr|,
        \qquad 
        \forall x, x' \in \RN^2,\, t \geq 0.
    \end{equation}
    Since $f$ is uniformly continuous and $\log a/\log N\leq\kappa$,
    it follows that the family
    \begin{equation}
        \left\{F_\varepsilon\!\left( \cdot, \frac{\log a}{\log N}\right)
        \, \middle |\, 
        N \geq 2,\ a \in [N^\delta,N^\kappa]\cap \ZN \right\}
    \end{equation}
    is uniformly bounded and equicontinuous on $\RN^2$.
    Hence, 
    it follows from Corollary~\ref{cor:sum to int}
    that, uniformly in $a \in [N^\delta,N^\kappa]\cap \ZN$,
    \begin{align}
        &\frac{1}{\log N}
        E\!\left[
            F_\varepsilon\!
            \left(
                \frac{S^1_a}{\sqrt{a}}, \frac{\log a}{\log N}
            \right)
            \mathbf{1}_D(\bm{S}_a)
        \right]\\
        &=
        \frac{c_*}{2a \log N}
        \int_{\RN^2}
        F_\varepsilon\!
        \left(
            x, \frac{\log a}{\log N}
        \right)
        \nu_{\Gamma/2}(dx)
        +\frac{o_{f,\delta}(1;N\to\infty)}{a\log N}.
    \end{align}
    Now, by the same argument as in the proof of Lemma~\ref{lem:linearize},
    we verify the desired result.
\end{proof}

We convert the sum over well-separated cluster-starting times
into an ordered time integral.
To this end,
we use a basic result on Riemann-sum approximations,
stated in Appendix~\ref{sec:Riemann sum}. 

\begin{lem} \label{lem:clust to int}
    For each $L \geq 1$, $\delta \in (0,1)$, and $\varepsilon \in (0,1/2)$,
    \begin{align}
        &\left(\frac{c_*}{2}\right)^L
        \sum_{(a_l)_{l=1}^L \in \CutTimes_L(\varepsilon)}
        \prod_{l=1}^{L} 
        \left\{
            \frac{1}{a_l \log N}
            \int_{\RN^2}
            F_\varepsilon\left( x, \frac{\log a_l}{\log N} \right) 
            \nu_{\Gamma/2}(dx)
        \right\}\\
        &=
        \frac{1}{L!}
        \left(
            \frac{c_*}{2}
            \int_\delta^\infty
            \int_{\RN^2}
            F_\varepsilon(x,t)\, \nu_{\Gamma/2}(dx)\, dt
        \right)^L
        +
        o_{f,\delta,L}(1; N \to \infty\, \triangleright\, \varepsilon \to 0).
    \end{align}
\end{lem}

\begin{proof}
    We set
    \begin{equation}
        H_\varepsilon(t) \coloneqq \int_{\RN^2} F_\varepsilon(x,t)\, \nu_{\Gamma/2}(dx).
    \end{equation}
    Note that $H_\varepsilon$ is continuous, is supported on a compact set, 
    and satisfies $\|H_\varepsilon\|_\infty \leq \|f\|_\infty$.
    Define
    \begin{align}
        S_{N,\varepsilon}
        &\coloneqq
        \sum_{(a_l)_{l=1}^L \in \CutTimes_L(\varepsilon)}
        \prod_{l=1}^{L}
        \frac{1}{a_l\log N}
        H_\varepsilon\Bigl(\frac{\log a_l}{\log N}\Bigr),
        \\
        U_{N,\varepsilon}
        &\coloneqq
        \sum_{\lceil N^\delta \rceil \leq a_1 < \cdots < a_L <\infty}
        \prod_{l=1}^{L}
        \frac{1}{a_l\log N}
        H_\varepsilon\Bigl(\frac{\log a_l}{\log N}\Bigr).
    \end{align}
    By symmetry,
    we have 
    \begin{equation}
        U_{N,\varepsilon}
        =
        \frac{1}{L!}
        \left(
            \sum_{a = \lceil N^\delta \rceil}^{\infty}
            \frac{1}{a\log N}
            H_\varepsilon\Bigl(\frac{\log a}{\log N}\Bigr)
        \right)^L
        +o_{f,\delta,L}(1;N\to\infty).
        \label{eq: symmetrized logarithmic sum}
    \end{equation}
    This, combined with Lemma~\ref{lem:log sum to int}, yields 
    \begin{equation} \label{eq: unrestricted logarithmic sum}
        \lim_{N\to\infty}U_{N,\varepsilon}
        =
        \frac{1}{L!}
        \left(
            \int_\delta^\infty H_\varepsilon(t)\,dt
        \right)^L.
    \end{equation}
    It remains to estimate the difference between $S_{N,\varepsilon}$ and $U_{N,\varepsilon}$.
    The support assumption on $f$ implies 
    \begin{equation}
        U_{N,\varepsilon}
        =
        \sum_{\lceil N^\delta \rceil \leq a_1 < \cdots < a_L \leq \lfloor N^\kappa \rfloor}
        \prod_{l=1}^{L}
        \frac{1}{a_l\log N}
        H_\varepsilon\Bigl(\frac{\log a_l}{\log N}\Bigr).
    \end{equation}
    Hence, when $L = 1$, we have $U_{N,\varepsilon} = S_{N, \varepsilon}$,   
    which, combined with \eqref{eq: unrestricted logarithmic sum}, yields the desired result.
    Assume $L \geq 2$.
    Then the difference between $S_{N,\varepsilon}$ and $U_{N,\varepsilon}$ is supported 
    on ordered tuples $(a_l)_{l=1}^L$ in $[N^\delta, N^\kappa] \cap \ZN$ 
    for which $a_l \leq a_{l-1}+\lfloor a_{l-1}^{1+\varepsilon} \rfloor$ for some $2\leq l\leq L$.
    We thus have
    \begin{align}
        |U_{N,\varepsilon}-S_{N,\varepsilon}|
        &\lesssim_{L,f}
        \sum_{i=2}^{L}
        \sum_{\substack{\lceil N^\delta \rceil \leq a_1 < \cdots < a_L \leq \lfloor N^\kappa \rfloor\\
        a_i\leq a_{i-1}+\lfloor a_{i-1}^{1+\varepsilon} \rfloor}}
        \prod_{l=1}^{L}\frac{1}{a_l\log N}.
    \end{align}
    For fixed $a_{i-1} \in [N^\delta, N^\kappa] \cap \ZN$,
    \begin{equation}
        \sum_{a_i = a_{i-1}+1}^{a_{i-1} + \lfloor a_{i-1}^{1+\varepsilon} \rfloor}
        \frac{1}{a_i\log N}
        \lesssim_\delta
        \frac{1}{\log N}
        \log\bigl(1+a_{i-1}^{\varepsilon}\bigr)
        =o_{f,\delta}(1; N \to \infty\, \triangleright\, \varepsilon \to 0).
    \end{equation}
    This yields, for each $i \in \{2,\ldots,L\}$,
    \begin{align}
        \sum_{\substack{\lceil N^\delta \rceil \leq a_1 < \cdots < a_L \leq \lfloor N^\kappa \rfloor\\
        a_i\leq a_{i-1}+\lfloor a_{i-1}^{1+\varepsilon} \rfloor}}
        \prod_{l=1}^{L}\frac{1}{a_l\log N}
        &\leq    
        \frac{o_{f,\delta}(1; N \to \infty\, \triangleright\, \varepsilon \to 0)}{(\log N)^{L-1}}
        \left(\sum_{n=1}^{\lfloor N^\kappa \rfloor} \frac{1}{n}\right)^{L-1}\\
        &= 
        o_{f,\delta,L}(1; N \to \infty\, \triangleright\, \varepsilon \to 0).
    \end{align}
    Therefore,
    \begin{equation}
        \lim_{\varepsilon \to 0}
        \limsup_{N\to\infty}
        |U_{N,\varepsilon}-S_{N,\varepsilon}|
        = 0.
    \end{equation}
    Combining this with \eqref{eq: unrestricted logarithmic sum} gives the desired estimate.
\end{proof}

We are now in a position to prove Theorem~\ref{thm:lap conv}.

\begin{proof}[Proof of Theorem~\ref{thm:lap conv}]
    By Proposition~\ref{prop:chaos}, Corollary~\ref{cor:clust approx}, and 
    Lemmas~\ref{lem:remove annu}--\ref{lem:clust to int},
    we have 
    \begin{align}
        E\!\left[e^{-\Pi_N^{(\delta)}(f)}\right]
        &= 
        1 +
        \sum_{L=1}^M
        \frac{1}{L!}
        \left(
            \frac{c_*}{2}
            \int_\delta^\infty
            \int_{\RN^2}
            F_\varepsilon(x,t)\, \nu_{\Gamma/2}(dx)\, dt
        \right)^L\\[0.4em]
        &\qquad
        +
        o_{f,\delta}(1; N \to \infty\, \triangleright\, \varepsilon \to 0\, \triangleright\, M \to \infty).
        \label{eq:lap conv1}
    \end{align}
    We have $F_\varepsilon(x, t) \to F_0(x,t)$ locally uniformly in $(x, t) \in \RN^2 \times [0,\infty)$.
    The support assumption on $f$ implies $F_\varepsilon(\cdot, t) = 0$ 
    for any $t \geq \kappa$ and $\varepsilon \in (0,1/2)$.
    Thus, we deduce that 
    \begin{equation}
        \lim_{\varepsilon \to 0} 
        \int_\delta^\infty
        \int_{\RN^2} F_\varepsilon(x,t)\, \nu_{\Gamma/2}(dx)\, dt
        = 
        \int_\delta^\infty
        \int_{\RN^2} F_0(x,t)\, \nu_{\Gamma/2}(dx)\, dt.
    \end{equation}
    Combining this with \eqref{eq:lap conv1} yields
    \begin{equation}
        \lim_{N \to \infty}
        E\!\left[e^{-\Pi_N^{(\delta)}(f)}\right]
        =
        \exp\!\left(
            \frac{c_*}{2}
            \int_\delta^\infty
            \int_{\RN^2} F_0(x,t)\, \nu_{\Gamma/2}(dx)\, dt
        \right).
    \end{equation}
    We now let $\delta \to 0$.
    By Proposition~\ref{prop:early col},
    \begin{equation} \label{eq:lap conv2}
        \lim_{N \to \infty} 
        E\!\left[e^{-\Pi_N(f)}\right]
        =\lim_{\delta\to 0} \lim_{N \to \infty} 
        E\!\left[e^{-\Pi^{(\delta)}_N(f)}\right]
        = 
        \exp\!\left(
            \frac{c_*}{2}
            \int_0^\infty
            \int_{\RN^2} F_0(x,t)\, \nu_{\Gamma/2}(dx)\, dt
        \right).
    \end{equation}
    The relation $2/c_* = \Charact$ following from \eqref{eq:Charact} and \eqref{eq:c_*} yields
    \begin{equation} \label{eq:lap conv3}
        \frac{c_*}{2} F_0(x,t)
        =
        - \frac{c_*}{2}\cdot \frac{\Charact f(x, t)}{f(x, t)\,t + \Charact}
        = 
        - \frac{f(x, t )}{f(x, t)\,t + \Charact}.
    \end{equation}
    Therefore, the desired result follows from \eqref{eq:lap conv2} and~\eqref{eq:lap conv3}.
\end{proof}

\section{Proof of Theorems~\ref{thm:main} and~\ref{thm:limit meas}} \label{sec:proof}

In this section, 
combining Theorem~\ref{thm:lap conv} 
and standard convergence criteria for random variables and random measures,
we complete the proofs of Theorems~\ref{thm:main} and~\ref{thm:limit meas}.

\begin{proof}[Proof of Theorems~\ref{thm:main} and~\ref{thm:limit meas}]
    Let $M(dt\, dx\, dw)$ be a Poisson random measure on $(0, \infty) \times \RN^2 \times [0, \infty)$ 
    with intensity measure $\Charact t^{-2} e^{-\Charact w/t}\, dt\, \nu_{\Gamma/2}(dx)\, dw$
    and write $\{(t_i,x_i,w_i)\}_{i\in I}$ for its atoms.
    Define a random measure on $\RN^2 \times [0,\infty)$ by
    $\Pi_* \coloneqq \sum_{i \in I} w_i \delta_{(x_i, t_i)}$.
    This indeed defines a locally finite random measure.  
    In fact, for every $T>0$, 
    one can verify $E[\Pi_*(\RN^2\times[0,T])]=T/\Charact < \infty$.
    Fix a compactly supported, non-negative, continuous function $f$ on $\RN^2\times[0,\infty)$.
    Since $\|f\|_\infty < \infty$,
    for all sufficiently small $u \geq 0$,
    the function $u f$ satisfies Assumption~\ref{assum:test}.
    Fix such a $u$.
    We have 
    \begin{equation}
        \Pi_*(uf)
        = 
        \sum_{i \in I} 
        w_i u f(x_i,t_i)
        = 
        {\int_0^\infty \int_{\RN^2} \int_0^\infty}
        u w f(x, t)\, M(dt\, dx\, dw).
    \end{equation}
    By a standard fact on the Laplace transforms of Poisson random measures 
    (see, for example, \cite[Lemma~3.1(i)]{Kallenberg_17_Random}),
    it follows that
    \begin{align}
        E[e^{-\Pi_*(uf)}]
        =
        \exp\!
        \left(
            -{\int_0^\infty \int_{\RN^2} \int_0^\infty} 
            \left( 
                1- 
                e^{-u w f(x, t)}
            \right)
            \Charact t^{-2} e^{-\Charact w/t}\, dt\, \nu_{\Gamma/2}(dx)\, dw
        \right).
    \end{align}
    Evaluating the integral with respect to $w\in[0,\infty)$,
    the right-hand side equals 
    \begin{equation}
        \exp\!\left(
            -\int_0^\infty \int_{\RN^2} \frac{u f(x, t)}{t u f(x, t) + \Charact}\, 
            \nu_{\Gamma/2}(dx)\,dt
        \right).
    \end{equation}
    Combining this with Theorem~\ref{thm:lap conv},
    we deduce that 
    \begin{equation}
        \lim_{N \to \infty} E[e^{-\Pi_N(uf)}]
        = 
        E[e^{-\Pi_*(uf)}].
    \end{equation}
    By linearity,
    we have $\Pi_N(uf) = u \Pi_N(f)$ and $\Pi_*(uf) = u \Pi_*(f)$.
    Hence, 
    the above convergence implies that 
    the Laplace transform of $\Pi_N(f)$ converges to that of $\Pi_*(f)$ 
    for all sufficiently small $u$.
    By \cite[Theorem~2]{Mukherjea_Rao_Suen_06_Moment},
    it follows that
    $\Pi_N(f) \xrightarrow{\mathrm{d}} \Pi_*(f)$.
    A standard characterization of distributional convergence of random measures 
    (cf.\ \cite[Theorem~4.11]{Kallenberg_17_Random})
    then yields that 
    $\Pi_N\xrightarrow{\mathrm d}\Pi_*$ in the vague topology.
    This completes the proof.
\end{proof}

\appendix
\section{A logarithmic Riemann-sum approximation}
\label{sec:Riemann sum}

In this appendix, we record the logarithmic Riemann-sum approximation
used in the proof of Lemma~\ref{lem:clust to int}.
For completeness, we include a proof.

\begin{lem} \label{lem:log sum to int}
    Let $\delta>0$ and  $f$ be a compactly supported, continuous function on $[0, \infty)$.   
    It then holds that 
    \begin{equation}
        \lim_{N \to \infty}
        \sum_{n = \lceil N^\delta \rceil}^\infty
        \frac{1}{n \log N} f\!\left( \frac{\log n}{\log N} \right)
        = 
        \int_\delta^\infty f(x)\, dx.
    \end{equation}
\end{lem}

\begin{proof}
    Take $t = t(f) > \delta$ sufficiently large so that $\supp(f)\subset[0,t)$,
    and set $m_N \coloneqq \lceil N^\delta\rceil$ and $M_N\coloneqq\lfloor N^t\rfloor$.
    Then
    \begin{equation}
        \sum_{n=m_N}^\infty
        \frac{1}{n\log N}f\!\left(\frac{\log n}{\log N}\right)
        =
        \sum_{n=m_N}^{M_N}
        \int_n^{n+1}
        \frac{1}{n\log N}f\!\left(\frac{\log n}{\log N}\right)\,dx.
    \end{equation}
    For $x\in[n,n+1]$, we have
    \begin{equation}
        \left|\frac{\log n}{\log N}-\frac{\log x}{\log N}\right|
        \leq \frac{1}{\log N},
    \end{equation}
    uniformly in $n \in \NN$.
    Hence, by the uniform continuity of $f$,
    for any $x \in [n, n+1]$,
    \begin{align}
        \left|
            \frac{1}{n}f\!\left(\frac{\log n}{\log N}\right)
            -
            \frac{1}{x}f\!\left(\frac{\log x}{\log N}\right)
        \right|
        &\leq
        \frac{1}{n}
        \left|
            f\!\left(\frac{\log n}{\log N}\right)
            -
            f\!\left(\frac{\log x}{\log N}\right)
        \right|
        +
        \|f\|_\infty\left|\frac{1}{n}-\frac{1}{x}\right|\\
        &=
        \frac{o_f(1)}{x}
        +
        O_f(1) x^{-2},
    \end{align}
    where the error is uniform in $n \in \NN$ and $x \in [n,n+1]$.
    It follows that
    \begin{align}
        &\left|
        \sum_{n=m_N}^\infty
        \frac{1}{n\log N}f\!\left(\frac{\log n}{\log N}\right)
        -
        \int_{m_N}^{M_N+1}
        \frac{1}{x\log N}f\!\left(\frac{\log x}{\log N}\right)\,dx
        \right|\\
        &\leq
        \frac{1}{\log N}
        \int_{m_N}^{M_N+1}
        \left(\frac{o_f(1)}{x}+O_f(1)x^{-2}\right)\,dx\\
        &=o_f(1;N\to\infty).
    \end{align}
    By the change of variables $y=(\log x)/(\log N)$,
    \begin{equation}
        \int_{m_N}^{M_N+1}
        \frac{1}{x\log N}f\!\left(\frac{\log x}{\log N}\right)\,dx
        =
        \int_{\frac{\log m_N}{\log N}}^{\frac{\log(M_N+1)}{\log N}}
        f(y)\,dy.
    \end{equation}
    Since, as $N \to \infty$,
    \begin{equation}
        \frac{\log m_N}{\log N}\longrightarrow\delta,
        \qquad
        \frac{\log(M_N+1)}{\log N}\longrightarrow t,
    \end{equation}
    we obtain the desired result.
\end{proof}

\section*{Declaration on the use of AI}
All mathematical arguments and proofs were developed and verified by the authors.   AI tools were used only for discussions on the presentation of the main results and for reviewing the proofs.  

\section*{Acknowledgements}
The authors would like to thank Dr David Croydon for his helpful advice.
 The first author was supported by JSPS KAKENHI Grant Numbers JP24K16937 and JP25K00911. 
 The second author was supported by JSPS KAKENHI Grant Number JP26KJ0336. 

\bibliographystyle{amsplain}
\bibliography{Noda_bibfile}
\end{document}